\documentclass{base_class}

\usepackage[
hmarginratio={1:1},    
vmarginratio={1:1},     
textwidth=420pt,        
heightrounded,        
]{geometry}

\usepackage[hidelinks]{hyperref}

\usepackage[textsize=tiny]{todonotes}

\newcommand{\relint}{\operatorname{relint}}

\newcommand{\Moe}{\mathrm{B}}
\newcommand{\A}{\mathsf{A}}
\newcommand{\Kc}{\mathsf{K}}

\newcommand{\Proj}{\mathsf{P}}

\newcommand{\diag}{\operatorname{diag}}

\usepackage{mathrsfs}

\renewcommand{\calP}{\mathscr{P}}
\renewcommand{\calK}{\mathscr{K}}

\renewcommand{\calL}{\mathscr{L}}

\newcommand{\V}{\mathsf{V}}
\newcommand{\svol}{\mu_\omega}

\renewcommand{\epsilon}{\varepsilon}

\title{An Intersection Product for the Polytope Algebra}

\author{Thomas Wannerer}

\address{Friedrich-Schiller-Universit\"at Jena, Fakult\"at f\"ur Mathematik und Informatik, Institut f\"ur Mathematik, Inselplatz 5, 07743 Jena, Germany}
\email{thomas.wannerer@uni-jena.de}
\subjclass[2020]{52B45;	52B05}

\begin{document}

\maketitle
\begin{abstract}
We introduce a new multiplication for the polytope algebra, defined via  the intersection of polytopes. After establishing the foundational properties of this intersection product, we investigate finite-dimensional subalgebras that arise naturally from this construction. 
These subalgebras can be regarded as volumetric analogues of the graded M\"obius algebra, which appears in the context of the Dowling--Wilson conjecture. We conjecture that they also  satisfy the injective hard Lefschetz property and the Hodge--Riemann relations, and we prove these in  degree one.
\end{abstract}

\section{Introduction}
 
 The polytope algebra, introduced in the late 1980s by McMullen, is a remarkable object at the crossroads of convex geometry, combinatorics, and algebraic geometry  \cite{McMullen:PolytopeAlgebra,FultonSturmsfels:Intersection,Brion:Polytope}. 
 In a landmark paper, McMullen \cite{McMullen:SimplePolytopes} used this framework to show that the number of  faces of  simple polytopes 
 is characterized by a short list of properties. The sufficiency of these properties was already known from an  ingenious construction due to  Billera and Lee \cite{BilleraLee:sufficiency}, while their necessity  was deduced by Stanley \cite{Stanley:Number} from deep results in algebraic geometry. The contribution of McMullen was to provide a  convex geometric proof. The proofs of  Stanley and McMullen have  served as prototypes and sources of inspiration for recent decisive progress in algebraic combinatorics, see, e.g., \cite{AHK:Hodge, Huh:ICM18,Huh:ICM22,BHMPW:Singular}. 

The polytope algebra of an $n$-dimensional real vector space $V$ $$\Pi_*(V)=\bigoplus_{k=0}^n \Pi_k(V)$$ can be defined as the group of polytopal chains, i.e.  integer linear combinations 
$$ \sum_{i=1}^m a_i \mathbf{1}_{P_i}, \quad a_1,\ldots, a_m \in \ZZ,$$
of indicator functions of polytopes in $V$, modulo translations. In particular, if $[P]:=[\mathbf{1}_P]$ denotes the class of a polytope, then $[x+P]=[P]$ for all translations $x\in V$.  The multiplication in the polytope algebra is uniquely determined by the equation
\begin{equation}\label{eq:McMullen}[P] * [Q] = [P+Q],\end{equation}
where $P+Q$ is the Minkowski sum.

In convex geometry, intersection is often viewed as an operation dual to Minkowski addition.
At the level of polytopal chains, intersection behaves in a straightforward way: $\mathbf 1_{P\cap Q}= \mathbf 1 _P \cdot \mathbf 1 _Q$. However, because  intersection depends on the relative position of $P$ and $Q$, it does not descend to a well-defined operation on the polytope algebra. One  way to address this defect is to average over all relative positions of $P$ and $Q$. Setting aside questions about the existence of the following integral, one is led to consider the definition
\begin{equation}\label{eq:prelim}[P]\cdot [Q]= \int_{V} [P\cap(x+Q)]\, dx.\end{equation}
The starting point for this work was the realization that this definition does, in fact, extend to a new multiplicative structure on the polytope algebra---one that captures non-trivial geometric and combinatorial phenomena, as illustrated by the results discussed below.

\subsection{Main results}
One, albeit minor, issue with the preliminary definition \eqref{eq:prelim} is the absence of a canonical choice of a Lebesgue measure $dx$ on $V$. To address this,  we denote by $\Dens(V)$ the  one-dimensional vector space   of densities on $V$,  consisting of all (including negative) translation-invariant Radon measures on $V$. Moreover, it will be convenient to introduce the following notation.  
Throughout this paper,  unless specified otherwise, all tensor products are over $\RR$.  As will be recalled in Section~\ref{sec:polyalg}, $\Pi_k(V)$ is isomorphic to $\ZZ$ for $k=0$ and  a vector space over $\RR$ for $k>0$.

\begin{definition}We define $\Pi^n(V)=\Pi_0(V)\otimes_\ZZ \Dens(V)$ and  $\Pi^{k}(V)=\Pi_{n-k}(V)\otimes \Dens(V)$  for every  integer $k\in\{0,\ldots, n-1\}$ and  
	$$ \Pi^*(V)= \bigoplus_{k=0}^n \Pi^k(V).$$
\end{definition}

Our first theorem asserts that the integral in \eqref{eq:prelim} is indeed well defined and briefly summarizes the main properties of the resulting multiplication, which we call the intersection  product in the polytope algebra. For more precise statements, the reader is directed to Theorems~\ref{thm:Alesker-prod} and \ref{thm:Poincare}.
 
\begin{theorem}
There exists a multiplication, uniquely determined by the equation
\begin{equation}\label{eq:def-Alesker} ([P]\otimes \mu)\cdot ([P']\otimes \mu')= \int_{V} [P\cap (x+ P')]\otimes \mu'\, d\mu (x), \end{equation}
that endows $\Pi^*(V)$ with the structure of a unital graded commutative algebra satisfying Poincar\'e duality. 
\end{theorem}

The intersection product is  compatible with an operation $$f^*\colon \Pi^*(V)\to \Pi^*(W)$$ that we call the pullback along a linear map  $f\colon W\to V$. If $f$ is injective, then the pullback of $[P]\otimes \mu$ 
 can be interpreted as an average of the fibers of $P$ under the canonical projection $V\to V/f(W)$. An analogous construction---based on a different addition---leads to the notion of the  fiber polytope, introduced by Billera and Sturmfels \cite{BilleraSturmfels:FiberPolytopes}. We elaborate on this connection in Remark~\ref{rmk:fiber} below.  
  It turns out that  the pullback can be defined for all linear maps, not just injective ones, but the definition in the general case is somewhat less straightforward.

The following theorem summarizes the main properties of the pullback.
\begin{theorem} 
	The pullback along a linear map $f\colon W\to V$ satisfies the following properties: 
	\begin{enuma}
		\item It is  a morphism of algebras when $ \Pi^*(V)$ and $\Pi^*(W)$ are equipped with the intersection product.
		\item The pullback is compatible with the grading, 
		$$f^*(\Pi^k(V))\subseteq \Pi^{k}(W).$$
		\item If $g\colon U\to W$ is another linear map, then $(f\circ g)^*= g^*\circ f^*$. 
	\end{enuma}	
\end{theorem}

\begin{remark}
For continuous translation-invariant valuations, Alesker introduced in \cite{Alesker:Fourier} a notion of pushforward along linear maps that can be regarded as dual to the pullback in the polytope algebra;  see Remarks~\ref{rem:pullback} and \ref{rmk:pullback2} below. There is another notable  analogy with the theory of valuations on convex bodies. Bernig and Faifman \cite{BernigFaifman:Polytopealgebra} show that, in a precise sense, McMullen's multiplication \eqref{eq:McMullen} is closely related to the convolution of smooth translation-invariant valuations, introduced by Bernig and Fu \cite{BernigFu:Convolution}. We expect a similar relationship between intersection product and Alesker product.  Remark~\ref{rmk:Alesker-product} below  offers an initial insight into this relationship.

It seems that the only significant formal difference between these two settings is  the existence of the Fourier transform  within Alesker's theory of smooth translation-invariant valuations.
A key property of this operation is that it  intertwines the Alesker product and the Bernig--Fu convolution. 
In light of the recent  explicit description of the Fourier transform obtained by  Faifman and Wannerer \cite{FaifmanWannerer:Fourier},  it appears unlikely that this powerful operation also exists  within  the context of the polytope algebra.
\end{remark}

If we fix a positive density $\vol$  on $V$, as we will do from now on,  we  abbreviate our notation to
$\alpha [P] := [P]\otimes \alpha \vol \in \Pi^*(V)$ for $\alpha\in \RR$. Moreover,  when our discussion involves the notion of positivity, such as in \eqref{eq:HR-inequality} below,
we will tacitly identify $\Pi^n(V)\cong \RR$ via $\alpha[\{0\}]\mapsto \alpha$.
\begin{definition}
	For each polytope $P\subseteq V$, we define 
	$\ell_P:= [P]_{n-1}\in \Pi^1(V).$ 
\end{definition}
The Euler--Verdier involution $\sigma$ is an algebra automorphism of   $\Pi^*(V)$  and is compatible with the pullback. It differs from McMullen's Euler map by the sign $(-1)^n$.   For every polytope $P$, $\sigma(\ell_P) = - \ell_{-P}$.

The intersection product in the polytope algebra  is capable of expressing deep geometric and combinatorial properties, as the theorems that follow will demonstrate. Our first result is an Alexandrov--Fenchel-type inequality.  For a more elementary formulation in terms of mixed volumes, the reader is directed to Theorem~\ref{thm:AF}.

\begin{theorem}\label{mthm:AF} Let $C_1,\ldots, C_{n-2}$ be  centrally symmetric polytopes in $V$ and define $\ell_{\mathbf C}=\ell_{C_1} \cdots \ell_{C_{n-2}}$. Then, for every $x\in \Pi^1(V)$ and every  polytope $Q\subseteq V$, the following inequality holds: 
	\begin{equation}\label{eq:HR-inequality} \left(\sigma(x)\cdot \ell_Q \cdot  \ell_{\mathbf C}  \right)^2 \geq \left( \sigma(x)\cdot x \cdot  \ell_{\mathbf C} \right) \left( \sigma(\ell_Q)\cdot \ell_Q \cdot  \ell_{\mathbf C} \right).\end{equation}
	
\end{theorem}
Notice that inequality  \eqref{eq:HR-inequality} can be interpreted as a statement about the symmetric bilinear form $$q_{\mathbf C}(x,y)=  \sigma(x)\cdot y \cdot  \ell_{\mathbf C}, \quad x,y\in \Pi^1(V).$$
Indeed, since $\sigma$ is an algebra automorphism and the polytopes $C_i$ are centrally symmetric, for all $x,y\in \Pi^1(V)$
$$  \sigma(x)\cdot y \cdot  \ell_{\mathbf C}=  \sigma(y)\cdot x \cdot  \ell_{\mathbf C}.$$

\medskip
 For every  simple polytope $P$ with nonempty interior, McMullen \cite{McMullen:SimplePolytopes} defined $\mathsf{M}_*(P)\subseteq \Pi_*(V)$ as the $\ZZ$-span of all weak Minkowski summands of $P$.  This subset is, in fact, a finite-dimensional subalgebra with truly remarkable algebraic  properties. These imply restrictions on the dimensions of the graded components $\mathsf{M}_k(P)$, which in turn impose necessary conditions on the  number of faces of $P$.

Also $\Pi^*(V)$ contains natural subalgebras, and these  too are connected to combinatorics. 
\begin{definition}\label{def:A} Let $E$ be a finite set of lines in $V^*$ that pass through the origin and are not contained in a single hyperplane.
	\begin{enuma}
		\item  We denote by $\A^*(E)=\bigoplus_{k=0}^n \A^k(E)$ the subalgebra of $\Pi^*(V)$ generated by $\ell_P$ for all polytopes $P$ such that each facet  conormal of $P$ belongs to a line in $E$. 	 
		\item We denote by $\A^*_+(E)$ the subalgebra generated only by those elements $\ell_P \in \A^*(E)$ such that $P$ is centrally symmetric.
		\item  $\Kc(E)\subseteq \A^1_+(E)$ is the open convex cone   of all $\ell_P\in \A^1_+(E)$ with the property that each line in $E$ contains a facet conormal of $P$. 
	\end{enuma}
	
\end{definition}

It is not difficult to see that  $\A^*(E)$ and $\A^*_+(E)$ are finite-dimensional algebras for the intersection product. 
Moreover, the dimension of $\A^k_+(E)$ admits a straightforward combinatorial interpretation. 
Indeed, let $\calL_k(E)$ denote the set of $k$-dimensional linear subspaces that can be obtained as sums $L_1+\cdots + L_k$ of the lines in $E$.  Then, 
$$ \dim \A^k_+(E)  = |\calL_k(E)|.$$

The Dowling--Wilson conjecture \cite{DowlingWilson:Slimmest} asserts that for every nonnegative integer $k\leq n/2$
$$ |\calL_k(E)|\leq |\calL_{n-k}(E)|.$$
After several partial  results, starting from the 1940s with papers by de~Bruijn--Erd\H{o}s \cite{deBruijnErdos} and Motzkin \cite{Motzkin:Lines}, this conjecture was finally resolved in the affirmative by Huh--Wang \cite{HuhWang:Enumeration}. Shortly afterwards, the matroid version of the  Dowling--Wilson conjecture was established in the landmark paper \cite{BHMPW:Singular}. In both papers, the strategy  is to establish an injective hard Lefschetz theorem  for the graded  M\"obius algebra.

Poincar\'e duality, the hard Lefschetz theorem, and the Hodge--Riemann relations are fundamental properties of the cohomology ring of a compact K\"ahler manifold. This  algebraic structure---colloquially described as a K\"ahler package---has been observed to arise in different areas of mathematics. In algebraic combinatorics in particular,  the central importance of this concept has recently become apparent, see, e.g., \cite{AHK:Hodge,HuhWang:Enumeration,Huh:ICM18,Huh:ICM22,BHMPW:Singular}.

As we will discuss in detail in Section~\ref{sec:finite-dim}, the graded M\"obius algebra of $E$ and  $\A^*_+(E)$ are closely related. 
Based on Theorem~\ref{mthm:AF} and   analogous results for smooth translation-invariant valuations on convex bodies \cite{BKW:Hodge,KotrbatyWannerer:MixedHR} and the graded M\"obius algebra \cite{HuhWang:Enumeration,BHMPW:Singular}, we propose the following conjecture. If true, it would directly imply the Dowling--Wilson conjecture.

\begin{conjecture}\label{conj:Kahler}
	Let  $k$ be a nonnegative integer satisfying  $ k\leq n/2$. Suppose that  $\ell_{C_0},\ldots, \ell_{C_{n-2k}}\in \Kc(E)$, and define $\ell_{\mathbf C}= \ell_{C_1}\cdots \ell_{C_{n-2k}}$.
	Then, the following statements hold:
	\begin{enuma}
		\item \label{item:HL}\emph{Injective hard Lefschetz property.} The linear map $$ \A^k(E)\to \A^{n-k}(E),\quad  
		x\mapsto x\cdot \ell_{\mathbf C},$$
		is injective.
		
		\item \label{item:HR}\emph{Hodge--Riemann relations.}	If $x\in \A^k(E)$ satisfies $x\cdot  \ell_{C_0} \cdot \ell_{\mathbf C}=0$, then 
		\begin{equation}\label{eq:HR-ineq}  \sigma(x) \cdot  x \cdot    \ell_{\mathbf C} \geq 0.\end{equation}
		Moreover, equality holds if and only if $x=0$. 
	\end{enuma}
\end{conjecture}

\begin{remark} 
At first glance, the restriction to centrally symmetric polytopes in Conjecture~\ref{conj:Kahler}\ref{item:HL}, concerning the injective hard Lefschetz property, may seem unnecessary or even unnatural. However, in the absence of central symmetry, there is no reason to expect the Hodge--Riemann form in \ref{item:HR} to be symmetric. A formulation of the Hodge--Riemann relations for general reference polytopes may well exist, but this lies in uncharted territory, and we currently have only limited evidence supporting such a generalization.  Out of caution, we have therefore stated the conjecture only for centrally symmetric polytopes. The role played by central symmetry in the Alexandrov--Fenchel inequality (Theorem~\ref{mthm:AF}) is likewise unclear at present.
\end{remark}

For $k=0$, the  above conjecture  reduces  to the statement
$$\ell_{\mathbf C} =  \ell_{C_1} \cdots \ell_{C_n}
>0,$$
which we will prove in Proposition~\ref{prop:prop-wtV}. 
With considerably more effort, based on Theorem~\ref{mthm:AF} and a  characterization of the equality case in \eqref{eq:HR-ineq},  we are able to show the following. 

\begin{theorem}\label{mthm:1}
Conjecture~\ref{conj:Kahler} holds for $k=1$. 
\end{theorem}

The above theorem directly implies  the Dowling--Wilson conjecture for $k=1$. A notable feature of our proof of Theorem~\ref{mthm:1}  is that it uses a lemma by Motzkin on  points  in the real projective plane that appears already in his work \cite{Motzkin:Lines} on a special case of the Dowling--Wilson conjecture.

\section{Preliminaries on convex geometry}

We collect here for later use  important definitions and results from convex geometry. For more information we refer the reader to the monograph by Schneider \cite{Schneider:BM}.

\subsection{Convex bodies} 

\label{sec:convex}

In this paper,  a convex body is a nonempty compact convex subset of $\RR^n$. We denote the set all convex bodies in $\RR^n$  by $\calK(\RR^n)$. The Hausdorff distance $\delta(K,L)$ turns this set into a locally compact metric space.
The intersection of convex bodies is in general not continuous. However, under the assumption that $K$ and $L$ cannot be separated by a hyperplane, one has the following result. Recall that two convex bodies can be separated by a hyperplane if there exists a linear functional $f\colon \RR^n\to \RR$ and a number $\alpha$ such that 
$K\subseteq \{ x\colon f(x)\leq \alpha \}$ and $L\subseteq  \{x\colon f(x)\geq \alpha\}$.

\begin{lemma}\label{lemma:convergence-intersection}Let $K, L$ be convex bodies in $\RR^n$ that cannot be separated
	by a hyperplane. If $K_i , L_i$ ($i\in \NN$) are convex bodies  with $K_i \to  K$ and
	$L_i \to L$ for $i \to \infty $, then $K_i \cap  L_i  \neq\emptyset$  for almost all $i$ and $K_i \cap  L_i \to K \cap  L$ for
	$i\to \infty$.
	\end{lemma}

A convex body $K\subseteq\RR^n$ is uniquely determined by its support function $h_K(x)= \sup\{\langle x,y\rangle\colon y\in K\}$, $x\in \RR^n$. Here $\langle x,y\rangle $ denotes the euclidean inner product. 
	
For any convex body $K\subseteq \RR^n$ there exists a convex body $\Proj K\subseteq \RR^n$ such that 
$$ h_{\Proj K} (u)= \vol(P_{u^\perp} K),  \quad u \in S^{n-1}.$$
Here $\vol$ denotes the Lebesgue measure in $u^\perp$,   $S^{n-1}\subseteq \RR^n$ is the euclidean unit sphere,  and $P_{u^\perp}$ is the orthogonal projection onto $u^\perp$. 
 $\Proj K$ is called the projection body of $K$.
 
The surface area measure $S_K$ of a convex body $K\subseteq \RR^n$ is a finite  Borel measure on the unit sphere such that 
$ S_{K}(\omega)$ is for any  Borel set $\omega\subseteq S^{n-1}$ the  $(n-1)$-dimensional Hausdorff measure of 
the set of all boundary points of $K$ at which there exists a
normal vector of $K$ belonging to $\omega$.  If $P$ is a polytope, then 
$$ S_{P}= \sum_{i=1}^m \vol_{n-1}(F_i) \delta_{u_i}$$
where $F_1,\ldots, F_m$ are the facets of $P$, $u_1,\ldots, u_m$ are the corresponding facet normals, and  $\delta_u$ denotes the Dirac measure.
 
Necessary and sufficient conditions for a finite  Borel measure to  be the surface area measure of a convex body are given by Minkowski's existence theorem: 
\begin{theorem}\label{thm:Minkowski}
A finite Borel measure $\mu$ on $S^{n-1}$ is the surface area measure of a convex body with non-empty interior if and only if 
$$ \int_{S^{n-1}} u\, d\mu(u)=0$$
and $\mu$ is not concentrated on an equator. Moreover, the equation $\mu= S_K$ determines convex bodies with nonempty interior uniquely up to translations.
\end{theorem} 
 
The Blaschke sum of two  convex bodies $K,L$ with nonempty interior is defined as the convex body $K\# L$  with centroid at the origin such that 
$ S_{K\# L}  = S_K  + S_L$. The existence and uniqueness  are a consequence of Theorem~\ref{thm:Minkowski}.
The Blaschke sum  is continuous with respect to the Hausdorff metric:
\begin{lemma}\label{lemma:cont-B} Let $K,L\subseteq \RR^n$ be convex bodies  with nonempty interior. If $K_i,L_i$, $i\in \NN$, are convex bodies such that  $K_i\to K$ and $L_i\to L$ for $i\to \infty$,  then $ K_{i}\# L_{i}$ exists for almost all $i$ and  $ K_{i}\# L_{i}\to  K\# L$ for $i\to \infty$. 
\end{lemma}

The support function of a projection body satisfies 
\begin{equation}\label{eq:projection-body} h_{\Proj K}(x) = \frac 1 2 \int_{S^{n-1}} |\langle u,x\rangle |\, dS_K(u).\end{equation}
 One consequence of  this formula is  that $\Proj (K\# L) = \Proj K + \Proj L$ for convex bodies with nonempty interior.

The mixed volume of $n$ convex bodies $K_1,\ldots, K_n$ in $\RR^n$ is denoted by $\V(K_1,\ldots, K_n)$. 
 The volume of a Minkowski linear combination of $m$ convex bodies is a homogeneous polynomial of degree $n$, 
$$  \vol(x_1K_1+ \cdots + x_m K_m) = \sum_{\substack{\alpha\in \NN^m\\\alpha_1+\cdots +\alpha_m=n}}  \binom{n}{\alpha} \V(K_1[\alpha_1], \ldots, K_m[\alpha_m])x^\alpha   ,$$
where  $x_1,\ldots, x_m$ are nonnegative numbers, 
$\binom{n}{\alpha}$ is the multinomial coefficient, 
$$\V(K_1[\alpha_1], \ldots, K_m[\alpha_m]) = \V(\underbrace{K_1,\ldots, K_1}_{\alpha_1 \text{ times}}, \ldots,\underbrace{K_m,\ldots,K_m}_{\alpha_m \text{ times}}),$$
and $x^\alpha= x_1^{\alpha_1}\cdots x_m^{\alpha_m}$.

\subsection{Polytopes}
A polytope in $\RR^n$ or, more generally, in a finite-dimensional real vector space $V$ is the convex hull of finitely many points. We denote by $\calP(V)$ the set of all polytopes in $V$.

Let $P\in \calP(V)$ be a polytope. A subset  $F\subseteq P$ is called a face of the  polytope $P$ if there exists a nonzero linear functional $f\colon V\to \RR$ such that $$P\subseteq \{ x\in V\colon f(x)\leq \alpha\}  \text{ and } F=P\cap \{x\in V\colon f(x)=\alpha\}.$$ 
If $F$ is a facet, then $f$ is called a facet conormal.
Moreover, $P$ itself is considered to be a (non-proper) face of $P$.
If $F$ is a face, then $\bar F$ denotes a linear subspace, namely the direction of the affine hull of $F$. 

 The conormal cone of a polytope $P\subseteq V$ at a nonempty face $F$ is 
 $$ N(F,P)= \{ \xi\in V^* \colon \langle \xi,v\rangle = h_P(\xi) \text{ for every } v\in F\}.$$
 Here $h_P\colon V^*\to \RR$ denotes the invariant version of the support function of $P$.
 The conormal cone of $P$ at a point $x\in P$ is 
 $$ N(x,P)= \{\xi\in V^* \colon \langle \xi,x\rangle = h_P(\xi)\}.$$

 We denote by $\relint A$ the relative interior of a convex set $A$, i.e., the interior of $A$ relative to the affine hull of $A$. 

The following  lemma describes the faces and conormal cones of the intersection of two polytopes.

\begin{lemma}\label{lemma:faces-intersection}Let $P,P'\subseteq V$ be polytopes. Then the following properties hold: 
	\begin{enuma}
		\item If $F$ is a face of $P$ and $F'$ is a face of $P'$, then $F\cap F'$ is a face of $P\cap P'$. 
		\item If $G$ is a nonempty  face of $P\cap P'$, then there exist   faces $F\subseteq P$ and $F'\subseteq P'$ such that  $G= F\cap F'$. 
	
		\item  Let  $F$ be a face of $P$ and $F'$ be a face of $P'$. If $\relint F \cap \relint F'\neq \emptyset$, then  the conormal cones satisfy
		$$ N(F\cap F',P\cap P')= N(F,P)+ N(F',P').$$
	\end{enuma}
\end{lemma}

 We will also need an analogous result for the intersection of a polytope  with a linear subspace. 

\begin{lemma}\label{lemma:faces-intersection-subspace}
	Let $V$ and $W$ be finite-dimensional vector spaces and let $f\colon W\to V$ be a linear injection. Let $P\subseteq V$ be a polytope. Then the following properties hold:
	
	\begin{enuma}
		\item If $F$ is a face of $P$, then $f^{-1}(F)$ is a face of $f^{-1}(P)$.
		\item If $G$ is a nonempty face of $f^{-1}(P)$, then there exists a face $F$ of $P$ such that $G= f^{-1}(F)$.
		\item Let $F$ be a face of $P$. If $f^{-1}(\relint F)\neq \emptyset$, then the conormal cones satisfy
		$$ N(f^{-1}(F), f^{-1}(P))= f^*(N(F,P)).$$ 
	\end{enuma}	
\end{lemma}

\section{Preliminaries on the polytope algebra}
\label{sec:polyalg}
Throughout this paper  $V$ denotes an $n$-dimensional real vector space.

\subsection{Definitions and basic properties}

 The polytope algebra $\Pi_*(V)$, which was introduced by McMullen in \cite{McMullen:PolytopeAlgebra}, is the free abelian group formally generated by  the polytopes in $V$ quotiented by the  subgroup generated by  
\begin{equation}\label{eq:rel1}P\cup Q +P\cap Q - P -Q,\end{equation}
whenever $P,Q\in \calP(V)$ are so that $P\cup Q$ is convex, and 
\begin{equation}\label{eq:rel2} (x+P) - P, \quad x\in V.\end{equation}  The equivalence class of a polytope $P$ in $\Pi_*(V)$ is  denoted by $[P]$. Moreover, we define $[\emptyset ] =0$. 

The polytope algebra enjoys a universal property: it linearizes valuations. Here a valuation with values in an abelian group $A$ is by definition  a function $\phi\colon \calP(V)\to A$ such that for all $P,Q\in \calP(V)$ 
$$ \phi(P\cup Q)= \phi(P)+\phi(Q)-\phi(P\cap Q),$$
provided $P\cup Q$ is convex. A valuation is called translation-invariant if $\phi(x+P)= \phi(P)$ holds for all $x\in V$ and all polytopes $P$. 

The universal property of the polytope algebra is an immediate consequence of a classical result of Groemer concerning the extension of valuations to convex chains (see, e.g., \cite[Theorem 2.2.1]{KlainRota}). 

\begin{proposition}[Universal property of $\Pi_*(V)$]\label{prop:universal}
	Let $\phi\colon \calP(V)\to A$ be a translation-invariant valuation with values in an abelian group $A$. Then there exists a unique morphism of abelian groups $\bar\phi \colon \Pi_*(V)\to A$ such that the following diagram commutes:
	\begin{center}\begin{tikzpicture}
			\matrix (m) [matrix of math nodes,row sep=3em,column sep=5em,minimum width=2em]
			{
				\calP(V)  &  A  \\
				\Pi_*(V)  &    \\ };
			\path[-stealth]
			(m-1-1) edge node [above] {$\phi$}  (m-1-2)
			(m-1-1) edge  (m-2-1)
			(m-2-1) edge node [below right] {$\bar \phi$} (m-1-2);
		\end{tikzpicture}
	\end{center}
	
\end{proposition}

 Minkowski addition endows $\Pi_*(V)$ with the structure of a commutative ring
\begin{equation}\label{eq:convolution} [P] * [Q]= [P+Q].\end{equation}
To distinguish this multiplication from the one we introduce in this paper, we will refer to \eqref{eq:convolution} as the \emph{convolution} in the polytope algebra.

Dilation of polytopes by nonzero numbers $\alpha\in \RR$ descends to a ring automorphism of $\Pi_*(V)$ that is uniquely determined by
$$\Delta(\alpha)[P] = [\alpha P].$$

The polytope algebra admits a natural grading that is compatible with dilations. We summarize the fundamental properties  of polytope algebra in the following theorem.

\begin{theorem}[{\cite[Theorem 1]{McMullen:PolytopeAlgebra}}]
As an abelian group, the polytope algebra admits a decomposition 
$$ \Pi_*(V)= \bigoplus_{k=0}^n \Pi_k(V)$$
 that satisfies the following properties:
 \begin{enuma}
 	\item The grading is compatible with convolution:
 	$$ \Pi_k(V) * \Pi_l(V) \subseteq \Pi_{k+l}(V),$$
 	where $\Pi_k(V)=\{0\}$ for $k>n$. 
 	\item $\Pi_0(V)\cong \ZZ$ is generated by the points $\{x\}$, $x\in V$. For $k\in \{1,\ldots, n\}$, $\Pi_k(V)$ is naturally a vector space over $\RR$.
 	\item For every positive real number $\alpha$ and $x\in \Pi_k(V)$
 	$$ \Delta(\alpha) x = \alpha^k x.$$
 	\item If $x,y\in \bigoplus_{k=1}^n\Pi_k(V)$ and $\alpha\in \RR$, then $(\alpha x)*y=x*(\alpha y)= \alpha(x*y)$. 
 \end{enuma}
\end{theorem}

Given the natural vector space structure on $\bigoplus_{k=1}^n \Pi_k(V)$, one may seek to characterize  its linear functionals.   A translation-invariant valuation $\phi\colon \calP(V)\to \RR$ is called dilation-continuous if for any polytope $P$ the function $\lambda \mapsto \phi(\lambda P)$, $\lambda>0$, is continuous. The following observation is due to McMullen \cite[Theorem 6.2]{McMullen:Separation}.

\begin{lemma}\label{lemma:dilation-cont} Let  $\phi\colon \calP(V)\to \RR$ be a translation-invariant valuation, and let $\bar\phi$ denote its extension to $\Pi_*(V)$. Then  $\phi$ is dilation-continuous if and only if the restriction of $\bar \phi$ to $\bigoplus_{k=1}^n \Pi_k(V)$ is a linear functional.
\end{lemma}

\subsection{The normal cycle embedding}
The polytope algebra can be embedded  into a vector space that seems easier to visualize than $\Pi_*(V)$. This embedding was first discovered by McMullen~\cite[Theorem 5]{McMullen:PolytopeAlgebra}, and below we give an invariant description of it. Under this embedding, the image of $[P]$ is essentially the normal cycle of $P$, a concept from geometric measure theory that  plays an important role in Alesker's theory of smooth valuations \cite{Alesker:VMfdsIII,Alesker:ValMfds,AleskerBernig:Product,Alesker:VMfdsI,Alesker:VMfdsII}. We therefore propose to call this embedding the normal cycle embedding.

We call a subset $C\subseteq V$ a  polyhedral cone if it is the intersection of finitely many halfspaces of the form  $\{v\in V\colon f(v)\leq 0\}$, where $f\colon V\to \RR$ is a linear functional. The cone group $\widehat\Sigma(V) $ is the free abelian group formally generated by polyhedral cones in $V$ quotiented by the subgroup generated by 
$$C\cup C' +C\cap C' - C -C',$$
whenever  $C\cup C'$ is convex, and 
$$ C, \quad \text{if } \dim C<\dim V.$$  The equivalence class of a polyhedral  cone  $C$ in $\widehat\Sigma(V)$ is denoted by $[C]$.

Recall that  a left Haar measure on a locally compact group is unique  up to normalization. 
 The one-dimensional vector space   $\Dens(V)$ of densities on $V$ consists of all (including negative) Haar measures on $V$. Alternatively,  $\Dens(V)$ can be defined as the vector space  of all functions 
 $\mu\colon \Lambda^n V\to \RR$ on the $n$-th exterior power of $V$ satisfying 
 $$  \mu(\alpha w)= |\alpha|\mu(w) \text{ for all } \alpha\in \RR\text{ and } w\in \Lambda^n V.$$
A density is called positive if $\mu(w)>0$ for all nonzero $w$. 

Recall that if $F$ is a nonempty face of a polytope $P$, then $\bar F$ denotes the direction of the affine hull of $F$. We  denote by $\epsilon_F$ the element of 
$\Dens(\bar  F)^*$ defined by 
$\mu\mapsto \mu(F)$.

Let us take the opportunity to comment on a canonical isomorphism $\Dens(V^*)\to \Dens(V)^*$. Since $V\times V^*$ carries a canonical symplectic structure, there is a canonical density $\mu_\omega$ on $V\times V^*$ called the Liouville measure. For any densities $\mu$ on $V$ and $\nu$ on $V^*$, the product measure $\mu\times \nu$ is proportional to $\mu_\omega$. This factor of proportionality defines  a canonical non-degenerate pairing $\Dens(V)\times \Dens(V^*)\to \RR$, which in turn defines the isomorphism 
\begin{equation} \label{eq:dual-dens-iso}\Dens(V^*)\to \Dens(V)^*.\end{equation}

We denote by 
$$L^\perp=\{ \xi\in V^*\colon \langle \xi,x\rangle =0 \text{ for all } x\in L\}$$ the  annihilator of a linear subspace $L\subseteq V$.

\begin{definition} We define
	$$ \Sigma(V)= \bigoplus_{k=0}^n \bigoplus_{L\in \Grass_k(V)}
 \Dens(L)^*\otimes_\ZZ  \widehat\Sigma(L^\perp),$$
	where the inner sum extends over all $k$-dimensional linear subspaces of $V$. For every nonempty polytope $P$, we call 
	$$ \nc(P)= \sum_F \epsilon_F \otimes_\ZZ [N(F,P)]\in \Sigma(V),$$
	where the summation extends over all nonempty faces of $P$, the normal cycle of $P$.
\end{definition}

\begin{remark} For every polytope $P$, $ \nc(P)$ can be regarded as a linear functional on translation-invariant smooth $(n-1)$-forms on the sphere bundle of  $\RR^n$.  Indeed, assuming  $V=\RR^n$ in what follows,  translation-invariant smooth $(n-1)$-forms on $\RR^n\times S^{n-1}$ can be identified with the elements of $$\bigoplus_{k=0}^{n-1} \Lambda^{k} (\RR^n)^*\otimes \Omega^{n-1-k}(S^{n-1}),$$
	where $\Omega^*(S^{n-1})$ denotes the space of smooth differential forms on the euclidean unit sphere.    For any  oriented linear subspace $L\subseteq \RR^n$, the elements of $\Dens(L)^*$ pair with  forms on  $L$ and therefore, after restriction to $L$, also with forms on $V$. 
	Consequently, for every face $F$ of  $P$ and every translation-invariant $(n-1)$-form $\omega$, one obtains a form $\langle \epsilon_F, \omega\rangle\in \Omega^{n-1-\dim F}(S^{n-1})$. Integration of these forms defines the linear functional
$$\nc(P)(\omega)= \sum_{F} \int_{N(F,P)\cap S^{n-1} } \langle \epsilon_F, \omega\rangle,$$
 where the sum extends over all proper faces of $P$ and the orientation of $\bar F^\perp$ is chosen so that  $\bar F\oplus \bar F^\perp= \RR^n$ has the standard orientation.    
It follows that $\omega\mapsto \nc(P)(\omega)$  coincides with the normal cycle of $P$  introduced by Fu \cite{Fu:Subanalytic,Fu:IGRegularity} when the latter is integrated against translation invariant forms.
\end{remark}

\begin{theorem}[{\cite[Theorem 5]{McMullen:PolytopeAlgebra}}]\label{thm:nc-embedding}
	The map $\nc\colon  \calP(V)\to \Sigma(V)$ extends to an injective map of abelian groups $ \nc\colon \Pi_*(V) \to    \Sigma(V)$. Moreover, for $k>0$ the restriction of $\nc$ to 
	$\Pi_k(V)$ is linear.
\end{theorem}

As an important corollary we obtain:
\begin{corollary}
$\Pi_n(V)$ is  canonically isomorphic to $\Dens(V)^*$. 
\end{corollary}

\section{Construction of the intersection product}
At first glance, the definition of the intersection product via the integral in \eqref{eq:def-Alesker} might seem problematic, as it appears to require a topology on the infinite-dimensional vector space $\Pi^*(V)$.
However, since every finite-dimensional real vector space carries a unique topology that makes it into a topological vector space, we need not concern ourselves with the topology on $\Pi^*(V)$---as long as all functions we consider take values in finite-dimensional subspaces of $\Pi^*(V)$.

The first few lemmas of this section take care of these technical issues. 
Recall that throughout this paper $V$ denotes an $n$-dimensional real vector space.

\begin{lemma}\label{lemma:generic}
	Let $P,P'\subseteq V$ be polytopes.  There exists an open set 
$U\subseteq V$ whose complement has measure zero and  such that  for  each  $x\in U$ and each pair of nonempty faces  $F$ of $P$ and $F'$ of  $P'$  the following property holds: If $F\cap (x+ F')\neq \emptyset$,  then
$$\relint F \cap (x+ \relint F') \neq \emptyset,$$ and 
$$\dim F + \dim F' \geq n.$$

\end{lemma}
\begin{proof} This follows from a standard reasoning, and we omit the proof for brevity.
\end{proof}

Throughout this paper, for a polytope $P\subseteq V$ and a density $\mu\in \Dens(V)$, we write, by  a slight abuse of notation, $$[P] \otimes \mu = [P]_0\otimes_\ZZ \mu  + \sum_{k=1}^n [P]_k \otimes \mu.$$

\begin{lemma}\label{lemma:finite-dim}Let $E=\{\xi_1,\ldots, \xi_m\}\subseteq V^*$ be  a finite set of nonzero linear functionals. Let $S_E\subseteq \Pi^*(V)$ denote  the subspace  spanned by all elements $[P]\otimes \mu$  where $P$ is a  polytope of the form 
	\begin{equation}\label{eq:P-conormal} P=\{ x\in V\colon \langle \xi_i, x\rangle \leq c_i \text{ for } i=1,\ldots, m\},\end{equation}
	for arbitrary real numbers $c_1,\ldots, c_m$. Then $S_E$ is  finite-dimensional. 
\end{lemma} 
\begin{proof}If $P$ is of the form \eqref{eq:P-conormal}, then each conormal cone $N(F,P)$ is generated by a subset of $E$. Consequently, the normal cycle of $[P]$ is contained in a finite-dimensional subspace of $\Sigma(V)$. Since $\nc\colon  \Pi^*(V)\to \Sigma(V)\otimes \Dens(V)$ is injective by Theorem~\ref{thm:nc-embedding}, the claim follows.
\end{proof}

\begin{lemma} \label{lemma:measurable}
	
	Let $E=\{\xi_1,\ldots, \xi_m\}\subseteq V^*$ be a finite set of nonzero linear functionals and let $S_E$ be defined as in Lemma~\ref{lemma:finite-dim}.  For all polytopes   $P=\{ x\in V\colon \langle \xi_i,x\rangle \leq c_i  \text{ for } i=1,\ldots, m\}$ and $P'=\{x\in V\colon \langle \xi_i,x\rangle \leq  c_i'  \text{ for } i=1,\ldots, m\}$ and all densities $\mu\in \Dens(V)$, the following properties hold:
	\begin{enuma}
		\item $[P\cap (x+P')]\otimes \mu\in S_E$ for all $x\in V$.
		\item The function $V\to S_E$, $x\mapsto   [P\cap (x+P')]\otimes \mu$, is measurable and essentially bounded.
		\item For each $\mu'\in \Dens(V)$, the integral
		$$ \int_V [P\cap (x+P')] \otimes \mu \,d\mu'(x) \in S_E $$ 
		is well defined. 
		\item The function $(0,\infty)\to S_E$, $\lambda \mapsto \int_V [(\lambda P)\cap (x+P')] \otimes \mu \,d\mu'(x)$, is continuous.
	\end{enuma}
	
\end{lemma} 
\begin{proof}
(a) is clear from the definition of $S_E$. 

 Let $U$ be as in  Lemma~\ref{lemma:generic}.  Let $x\in U$ and $\lambda_0>0$. By Lemma~\ref{lemma:faces-intersection}, the faces of  $\lambda P\cap(x+P')$ can be represented in the form $\lambda F\cap(x+F')$, where $F$ is a face of $P$ and $F'$ is a face of $P'$.   There are open  neighborhoods  $U'\subseteq U$ of $x$  and $U''$ of $\lambda_0$ such that for all $y\in U'$ and $\lambda\in U''$ and all faces $F$ and $F'$ the following property holds: 
$$ \lambda F \cap (y+ F') \neq  \emptyset\quad \Leftrightarrow \quad F \cap (x+ F') \neq  \emptyset.$$ 
 By Lemma~\ref{lemma:faces-intersection}, the conormal cones 
$N(\lambda F \cap (y+ F'),\lambda P\cap(y+P'))$  coincide for all $y\in U'$. Since by Lemma~\ref{lemma:convergence-intersection} the map $(\lambda, y)\mapsto \lambda F \cap (y+ F')$ is continuous in the Hausdorff metric, we conclude that 
$$ U'\times U''\ni (y,\lambda)\mapsto \nc(\lambda P\cap (y+P'))\in \nc(S_E)\subseteq \Sigma(V)\otimes \Dens(V)$$
is continuous. Consequently, as the complement of $U$ has measure zero, the map
 $V\to \nc(S_E)$, $x\mapsto\nc(P\cap (x+P'))$, is measurable and essentially bounded.    Using that $\nc\colon \Pi^*(V)\to \Sigma(V)\otimes \Dens(V)$ is a linear injection, this finishes the proof of (b).   We similarly obtain that $\lambda\mapsto  [(\lambda P)\cap (x+P')]$ is continuous for a.e.\ $x\in V$.

(c) is an immediate consequence of (a) and (b). Dominated convergence implies (d).
\end{proof}

The following proposition shows  that the integral in Lemma~\ref{lemma:measurable} is independent of the choice of subspace $S_E$. 
\begin{proposition}\label{prop:existence} Let $[P]\otimes \mu$ and $[P']\otimes \mu'\in \Pi^*(V)$. There exists a unique element $I\in \Pi^*(V)$ such that for every linear functional $\phi\colon \Pi^*(V)\to \RR$ the identity
	\begin{equation}\label{eq:integral} \langle \phi, I\rangle = \int_V \langle \phi,[P\cap (x+P')] \otimes \mu \rangle \,d\mu'(x)\end{equation}
	holds.
We denote this element by $\int_V [P\cap (x+P')] \otimes \mu \,d\mu'(x)$. 
\end{proposition}
\begin{proof}
For every linear functional $\phi$, it follows from Lemma~\ref{lemma:measurable} that $$V\to \RR, \quad x\mapsto  \langle \phi,[P\cap (x+P')] \otimes \mu \rangle$$
is measurable and essentially bounded.  Consequently, the integral in \eqref{eq:integral} is well-defined. If $I$ exists, then it must be unique. To establish existence, choose any $E\subseteq V^*$ as in Lemma~\ref{lemma:measurable} and define $I:=\int_V [P\cap (x+P')] \otimes \mu \,d\mu'(x)\in S_E$.  A routine argument shows that $I$ satisfies \eqref{eq:integral}.
\end{proof}

Before stating the main theorem of this section, we  describe the action of  $\GL(V)$ on the polytope algebra. This action is the standard one, where the general linear group acts on polytopes, extended to the polytope algebra.
 More precisely, by the universal property of the polytope algebra  (Proposition~\ref{prop:universal}),  there exists for every $g\in \GL(V)$ a unique linear transformation 
$L_g\in \GL(\Pi^*(V))$ such that 
	$$ L_g ([P]\otimes \mu)= [g P] \otimes g_{*}\mu$$ holds for every $ [P]\otimes \mu \in \Pi^*(V)$.
Here  $g_*\mu$ denotes the pushforward of the measure $\mu$ under the linear transformation $g$. For nonzero $\alpha$, we define dilation on $\Pi^*(V)$ by $\Delta(\alpha)= L_{\alpha \id}$.

\begin{theorem}\label{thm:Alesker-prod}
There exists a multiplication that endows $\Pi^*(V)$ with the structure of a commutative algebra, uniquely determined by the equation \eqref{eq:def-Alesker}.
This operation is called the \emph{intersection product} and  satisfies the following additional properties:
\begin{enuma}
	\item \label{item:grading} It is compatible with the grading: \[\Pi^{k}(V)\cdot \Pi^{l}(V)\subseteq \Pi^{k+l}(V).\]
	\item  The canonical element $\mu^* \otimes \mu \in \Dens(V)^*\otimes  \Dens(V)=\Pi^0(V)$, which we denote by $e_V$, is the identity.	
	\item \label{item:GL}  It is equivariant under the natural action of the general linear group.

\end{enuma}
\end{theorem}

\begin{proof}
We first establish existence. Fix $\mu'\in \Dens(V)$ and a polytope $P'$. By Proposition~\ref{prop:existence}, the map 
$\calP(V)\times \Dens(V)\to \Pi^*(V)$,
$$ \phi_{P',\mu'}(P,\mu) =\int_{V} [P\cap (x+ P')]\otimes \mu' \, d\mu (x)$$
is well defined. Since it is a translation-invariant valuation as a function of $P$, by the universal property of the polytope algebra (Proposition~\ref{prop:universal}), we obtain a map $\phi_{P',\mu'}\colon \Pi_*(V)\times \Dens(V)\to \Pi^*(V)$. 
Lemma~\ref{lemma:dilation-cont} implies that its restriction to $\bigoplus _{k\geq 1} \Pi_{k}(V)\otimes \Dens(V)$  is bilinear. Therefore it admits an extension $$\phi_{P',\mu'}\colon \Pi^*(V)\to \Pi^*(V).$$

Let  $x=\sum_{i=1}^m  [P_i]\otimes \mu_i$ be fixed. One immediately verifies that 
$$P'\mapsto \phi_{P',\mu'}(x)= \sum_{i=1}^m \phi_{P',\mu'}(P_i,\mu_i)$$ is a  dilation-continuous translation-invariant valuation. Again by the universal properties of the polytope algebra and the tensor product, we obtain an extension to $\Pi^*(V)$. We thus have constructed a map
$$ \Pi^*(V)\otimes \Pi^*(V)\to \Pi^*(V)$$ 
that extends \eqref{eq:def-Alesker}. This finishes the construction of the intersection product. Notice that equation \eqref{eq:def-Alesker} implies that the intersection product is  associative and commutative.

 We claim that for any $\lambda\neq  0$ and $x,y\in \Pi^*(V)$
 \begin{equation}\label{eq:dilation} \Delta(\lambda)x \cdot \Delta(\lambda)y=\Delta(\lambda) (x\cdot y).\end{equation}
Indeed, it suffices to prove this for generators, where it is straightforward to verify using \eqref{eq:def-Alesker}.
Let $x\in \Pi^k(V)$ and $y\in \Pi^l(V)$ and $\lambda>0$. Then $\Delta(\lambda)x =\lambda^{-k} x$ and $\Delta(\lambda)y =\lambda^{-l} y$. In combination with  equation \eqref{eq:dilation}, we obtain
$$ \Delta(\lambda)(x\cdot y)=  \lambda^{-(k+l)} \cdot (x\cdot y).$$
Hence $x\cdot y\in \Pi^{k+l}(V)$, as desired. This proves (a).

For (b), observe that $[\lambda P]\otimes \mu$, $\lambda>0$, lies in a finite-dimensional subspace of $\Pi^*(V)$. Hence we can obtain $[P]_n\otimes \mu$ as the limit of $\lambda^{-n} [\lambda P]\otimes \mu$ as $\lambda\to\infty$. A straightforward computation using \eqref{eq:def-Alesker}   shows that 
$$ ([P]_n\otimes  \mu)\cdot  ([P']\otimes \mu')=  \mu(P)\cdot ( [P']\otimes \mu')
$$ 
Choosing $P$ and $\mu$ so that $\mu(P)=1$, we obtain $e_V\cdot ([P']\otimes \mu')=  [P']\otimes \mu'$.

Finally, we consider the natural action of $\GL(V)$ on $\Pi^*(V)$.  It suffices to verify 
$ (L_g x)\cdot (L_g y)= L_g(x\cdot y)$  on the generators. In this case, the desired equation
 is an immediate consequence of \eqref{eq:def-Alesker}.
\end{proof}

\begin{remark}\label{rmk:Alesker-product} 
The intersection product is closely related to the product of smooth translation-invariant valuations, introduced by Alesker in \cite{Alesker:Product}. Indeed, if $A, A'$ are convex bodies in $\RR^n$ with a smooth   and strictly positively curved boundary, then 
$$\phi(K)=\vol(K+ A) \quad \text{and}\quad  \phi'(K)=\vol(K+A')$$
are smooth translation-invariant valuations. An application of Fubini's theorem shows that 
their Alesker product is the valuation
$$ K\mapsto  \int_{\RR^n} \vol(K+ (A\cap (x+A'))) dx.$$

\end{remark}

\begin{definition}The Euler--Verdier involution on $\Pi^*(V)$ is defined by   
	$$ \sigma|_{\Pi^k(V)}  = (-1)^k \Delta(-1)  . $$    
\end{definition}

\begin{remark}
The Euler--Verdier involution was introduced for smooth valuations  by Alesker \cite{Alesker:VMfdsII}, and it is closely  related to the Verdier duality of constructible functions. In the context of the polytope algebra, the Euler--Verdier  involution coincides with McMullen's Euler map up to  the factor $(-1)^n$. 
As a consequence of this modification, the Euler--Verdier involution commutes with the pullback. Moreover, this choice of sign is convenient for the formulation of the Hodge--Riemann relations in Conjecture~\ref{conj:Kahler}.
\end{remark}

\begin{corollary}
$\sigma$ is an algebra  automorphism of  $\Pi^*(V)$. 
\end{corollary}
\begin{proof}
This follows immediately from Theorem~\ref{thm:Alesker-prod}\ref{item:grading} and \ref{item:GL}.
\end{proof}

If $x$ and $y$ have complementary degrees,  intersection product and McMullen convolution are closely related.
\begin{lemma}\label{lemma:pairing}
Fix a euclidean inner product on $V$ to identify $V\cong \RR^n$ and $\Dens(V)\cong \Dens(V)^*\cong \RR$. 
Then 
$$ (x\cdot y)_0 = ( x * \Delta(-1)y)_n$$ holds
for all 	 $x,y\in \Pi^*(V)$.

\end{lemma} 
\begin{proof}
Again, it suffices to prove the desired identity on the generators. Let $\chi\colon \Pi_0(V)\to \RR$ denote the  Euler characteristic and let $\vol \colon \Pi_n(V)\to \RR$ be the restriction of the Lebesgue measure to polytopes. With this notation, the lemma follows from
\begin{align*} \chi([P]\cdot [Q]) & = \int_V \chi( P\cap(x+Q)) dx = \vol(P+ (-Q))\\
	&= \vol([P]* \Delta(-1)[Q]).
	\end{align*}
\end{proof}

\begin{theorem}[Poincar\'e duality] \label{thm:Poincare} For every nonzero $x\in \Pi^k(V)$ there exists $y\in \Pi^{n-k}(V)$ such that 
	$ x\cdot y \neq 0$. 
\end{theorem}
\begin{proof}
Theorem~11 of  \cite{McMullen:PolytopeAlgebra} asserts  that for every nonzero $x\in \Pi_k(V)$ with  $k\in \{0,\ldots, n-1\}$,  there exists $y_1\in \Pi_{1}(V)$	such that $x*y_1\neq 0$. Applying this result repeatedly shows the existence of an element $y\in\Pi_{n-k}(V)$ such that 
$x*y\neq 0$. From Lemma~\ref{lemma:pairing}  we deduce that $(x\otimes \mu) \cdot \Delta(-1) (y\otimes \mu)\neq 0$ for any $\mu\neq 0$.

\end{proof}

\section{Construction of the pullback along linear injections}

We construct the  pullback along a linear map $f\colon W\to V$ in two stages. In this section, we consider the special case where $f$ is a linear injection. We address the  general case in Section~\ref{sec:pull-gen}, once we have established the fundamental properties of the exterior product.  Throughout  this section,  $W$ denotes  a  finite-dimensional real vector space.

\begin{lemma}\label{lemma:generic-subspace}Let $f\colon W\to V$ be a linear injection and let $P\subseteq V$ be a polytope. There exists an open set $U\subseteq  V/f(W)$ whose complement has measure zero such that  for each  face  $F$ of $P$ and each $[x] \in U$ the following properties hold: If 
$f^{-1}(x+F)\neq \emptyset$, then
$$ f^{-1}(x+\relint F)  \neq \emptyset$$  and  $$\dim F\geq \dim V-\dim W.$$

\end{lemma}
\begin{proof}
As was the case for Lemma~\ref{lemma:generic}, the proof is straightforward and is therefore omitted.
\end{proof}

\begin{lemma}\label{lemma:measurable-subspace}
	
	Let $E=\{\xi_1,\ldots, \xi_m\}\subseteq V^*$ be a finite set of nonzero linear functionals and let $S_{f^*E}$ be defined as in Lemma~\ref{lemma:finite-dim} with $$f^*E=\{ f^*\xi_i\colon i=1,\ldots,m\}\setminus\{0\}.$$   For all polytopes of the form   $P=\{ x\in V\colon \langle \xi_i,x\rangle \leq c_i \text{ for }i=1,\ldots, m\}$ and every $\mu_W\in \Dens(W)$ the following properties hold:
	\begin{enuma}
		\item $[f^{-1}(x+P)]\otimes \mu_W \in S_{f^*E}$ for all $x\in V$.
		\item The function $V/f(W)\to S_{f^*E}$, $[x]\mapsto   [f^{-1}(x+P)]\otimes \mu_W$, is measurable and essentially bounded.
		\item For each density $\mu_{V/f(W)}\in \Dens(V/f(W))$, the integral
		$$ \int_{V/f(W)} [f^{-1}(x+P)]\otimes \mu_W \, d\mu_{V/f(W)} ([x])\in S_{f^*E} $$ 
		is well-defined. 
		\item The function $(0,\infty)\to S_{f^*E}$, $\lambda \mapsto \int_{V/f(W)} [f^{-1}(x+\lambda P)]\otimes \mu_W \, d\mu_{V/f(W)} ([x])$, is continuous.
	\end{enuma}
	
\end{lemma}

\begin{proof} We omit the proof, because it is parallel to  the proof of Lemma~\ref{lemma:measurable}.
\end{proof}

As in Proposition~\ref{prop:existence},   the integral 
$$\int_{V/f(W)} [f^{-1}(x+P)]\otimes \mu_W \, d\mu_{V/f(W)} ([x])$$ can be defined in a way that is independent of the choice of subset $E\subseteq V^*$. 

 Before we state the main result of this section, we remind the reader of a useful isomorphism for densities.

\begin{lemma}\label{lemma:iso-dens} Let $0 \rightarrow X\stackrel{i}{\rightarrow}Y\stackrel{\pi}{\rightarrow}Z\rightarrow0$ be an exact sequence of finite-dimensional real vector spaces.  There exists a canonical isomorphism $$ \Dens(Y)\to \Dens(X)\otimes \Dens(Z)$$ so that $\mu$ is mapped to an element $\mu_X\otimes \mu_{Z}$ satisfying 
	\begin{equation}\label{eq:change-of-variables}\int_Y h\, d\mu= \int_Z \int_{X} h(i(x)+ y) d\mu_X(x)d\mu_Z(z), \quad y\in \pi^{-1}(z),\end{equation}
	for compactly supported continuous functions $h\colon Y\to \RR$. 
\end{lemma}
\begin{proof}To see the existence of an isomorphism with the desired properties just note that 
	$$ \int_Z \int_{X} h(i(x)+ y) d\sigma(x)d\tau(z), \quad y\in \pi^{-1}(z),$$
	defines an element of $\Dens(Y)$ for all densities $\sigma\in \Dens(X)$ and $\tau\in \Dens(Z)$. 	
\end{proof}

The following result  collects the fundamental  properties of the pullback along a linear injection.

\begin{theorem} \label{thm:pullback}
	Let $f\colon W\to V$ be a linear injection. There exists a linear map 
	$f^*\colon \Pi^*(V)\to \Pi^*(W)$, called the \emph{pullback} along $f$, that is uniquely determined by the equation
	\begin{equation}\label{eq:def-pullback} f^*([P]\otimes \mu) = \int_{V/f(W)} [f^{-1}(x+P )]\otimes \mu_W\, d\mu_{V/f(W)}([x]) ,\end{equation}
	where $\mu=\mu_W \otimes \mu_{V/f(W)}$ under the canonical isomorphism of Lemma~\ref{lemma:iso-dens}.
	The pullback satisfies the following additional properties:
	\begin{enuma}
		\item $f^*$ is a morphism of algebras when $ \Pi^*(V)$ and $\Pi^*(W)$ are equipped with the intersection product.
		\item \label{item:pull-grad} $f^*(\Pi^k(V))\subseteq \Pi^{k}(W)$.  
		\item \label{item:contra} If $g\colon U\to W$ is another linear injection, then $(f\circ g)^*= g^*\circ f^*$. 
		\item $f^*$ commutes with the Euler--Verdier involution.
	\end{enuma}	
\end{theorem}

\begin{proof}[Proof of Theorem~\ref{thm:pullback}] Since uniqueness is clear, it suffices to prove existence. For any $\mu=\mu_W \otimes\mu_{V/f(W)}  \in \Dens(V) 
	$ and any polytope $P\subseteq V$ define 
	$$ \phi(P,\mu)= \int_{V/f(W)} [f^{-1}(x+P)]\otimes \mu_W d\mu_{V/f(W)} ([x]).$$
	Lemma~\ref{lemma:measurable-subspace} guarantees that this is well defined. Observe that the map $\calP(V)\to \Pi^*(W)$, $P\mapsto \phi(P,\mu)$ is a translation-invariant valuation. By the universal property of polytope algebra, there exists a unique extension $\Pi_*(V)\to \Pi^*(W)$.  Since $(x,\mu)\mapsto \phi(x,\mu)$, $x\in \bigoplus_{k\geq 1} \Pi_k(V)$,  is bilinear by virtue of Lemma~\ref{lemma:dilation-cont}, we have constructed a linear map $f^*\colon \Pi^*(V)\to \Pi^*(W)$ with the desired property \eqref{eq:def-pullback}. 
	
	To prove (a) we  verify $f^*(x\cdot y) = f^*x \cdot f^*y$ for generators and then show that the identity is mapped to the  identity. Using Lemma~\ref{lemma:iso-dens} applied to $\mu'$ and $0\to W\to V\to V/f(W)\to 0$  
	in the fourth equality, we compute 
	\begin{align*} 
	 &f^*( ([P] \otimes  \mu) \cdot ([P']\otimes \mu'))\\ &  =  f^* \left( \int_{V} [ P \cap (x + P')] \otimes \mu \; d\mu'(x)\right)		\\
	& =  \int_{V} f^* ([ P \cap (x + P')] \otimes \mu)  d\mu'(x)		\\
	& =  \int_{V}  \left( \int_{V/f(W)}  [ f^{-1}(y+P) \cap f^{-1}(x+y + P')]\otimes \mu_W\, d\mu_{V/f(W)}([y])  \right)  d\mu'(x)		\\
	& =  \int_W \int_{V/f(W)} \left( \int_{V/f(W)}  [ f^{-1}(y+P) \cap f^{-1}(f(w)+y' + P')] \otimes \mu_W\, d\mu_{V/f(W)}([y])\right)
	\\ & \quad d\mu'_{V/f(W)}([y'])  d\mu_W'(w)		\\
	& = \int_{V/f(W)}\int_{V/f(W)}\left( \int_W   [ f^{-1}(y+P) \cap (w+f^{-1}(y' + P'))]\otimes \mu_W\,d \mu'_W(w) \right)  \\
	& \quad d\mu_{V/f(W)}([y]) d\mu'_{V/f(W)}([y']) \\
	 & =\int_{V/f(W)}\int_{V/f(W)}   ([ f^{-1}(y+P)]\otimes  \mu_W)\cdot ([f^{-1}(y' + P')]\otimes \mu'_W) \\
	 & \quad   d\mu_{V/f(W)}([y]) d\mu'_{V/f(W)}([y']) \\ 
	 & = f^*([P]\otimes \mu) \cdot f^*([P']\otimes \mu'). 
	\end{align*}

Let us write $n=\dim V$ and $m=\dim W$. Recall that the identity element may be expressed as $e_V=[P]_n \otimes \mu$ with $\mu(P)=1$. Also recall that $[P]_n = \lim_{\lambda \to \infty} \lambda^{-n} [\lambda P]$. Using this, we obtain
\begin{align*}
f^*(e_V) & = \lim_{\lambda\to \infty} \lambda^{-n} \int_{V/f(W)} [f^{-1}(y+\lambda P)]\otimes \mu_W\, d\mu_{V/f(W)}([y])\\
& = \lim_{\lambda\to \infty} \lambda^{-m} \int_{V/f(W)} [\lambda f^{-1}(z+P)]\otimes \mu_W\, d\mu_{V/f(W)}([z])\\
& =\int_{V/f(W)} [ f^{-1}(z+P)]_m \otimes \mu_W \,d\mu_{V/f(W)}([z]).
\end{align*}
Observe that 
$$ 
\int_{V/f(W)} \mu_W( f^{-1}(z+P)) d\mu_{V/f(W)}([z]) = \mu(P)=1$$
by  \eqref{eq:change-of-variables}. Hence $f^*(e_V) = e_W$ is the identity in $\Pi^*(W)$. 

To prove \ref{item:pull-grad}, observe that for any $\lambda>0$ and any $x\in \Pi^*(V)$
$$ f^*(\Delta(\lambda) x)=  \Delta(\lambda) f^*(x).$$
Indeed, this identity is obvious in the case of generators $x= [P]\otimes \mu$. If $x\in \Pi^k(V)$, it follows that $\Delta(\lambda) f^*(x)= \lambda^{-k} f^*(x)$. Thus $f^*(x)\in \Pi^{k}(W)$, as  required.

To prove \ref{item:contra} observe that by the definition of the pullback 
\begin{align*}  g^*(f^*([P]\otimes & \mu))=\\& 
	\int_{V/f(W)} \int_{W/g(U)}  [ g^{-1}(f^{-1}(P +x +f(y)))] \otimes \mu_U \, d\mu_{W/g(U)}([y]) d\mu_{V/f(W)}([x]).
	\end{align*}
Let $h\colon V\to \RR$ be a compactly supported continuous function, and define 
$$ h_U([v])= \int_{U} h(f(g(u)) + v) d\mu_U(u), \quad [v]\in V/f(g(U)).$$
On the one hand,  applying 
Lemma~\ref{lemma:iso-dens} to $0\to W/g(U)\to V/f(g(U)) \to V/f(W)\to 0$ yields a  density $\wt \mu_{V/f(g(U))}$ such that 
\begin{align*} \int_{V/f(W)} \int_{W/g(U)} & h_U([f(y) + x])  d\mu_{W/g(U)}([y]) d\mu_{V/f(W)}([x])\\
& = \int_{V/f(g(U))}  h_U([v]) d\wt \mu_{V/f(g(U))}([v]).\end{align*}
On the other hand, applying Lemma~\ref{lemma:iso-dens} to $0\to U\to V\to V/f(g(U))\to 0$  yields 
\begin{align*}  \int_{V/f(W)} \int_{W/g(U)} &  h_U([f(y) + x]) d\mu_{W/g(U)}([y]) d\mu_{V/f(W)}([x]) \\ & =\int_V h d\mu \\
	& = \int_{V/f(g(U))} h_U([v]) d\mu_{V/f(g(U))}([v]) .\end{align*}
 Hence $\wt \mu_{V/f(g(U))} =  \mu_{V/f(g(U))}$ and we conclude that 
\begin{align*} g^*(f^*([P]\otimes \mu)) & = \int_{V/f(g(U))} \int_U [ (f\circ g)^{-1}(P +v)] \otimes \mu_U \,   d\mu_{V/f(g(U))}([v]) \\ 
	& =
(f\circ g)^*([P]\otimes \mu).
\end{align*}
This finishes the proof of \ref{item:contra}. 

Finally, that the pullback commutes with the Euler--Verdier involution is an immediate consequence of \ref{item:pull-grad} and \ref{item:contra}.
\end{proof}

\begin{remark}
	\label{rmk:fiber}
	There is a similar-looking, but essentially different construction known as the fiber polytope. For comparison, let $W$ be a linear subspace of $V$. Let $f\colon W\to V$ denote the inclusion  and let $\pi\colon V\to V/W$ denote the canonical projection.  Let $P$ be a polytope in $V$ and put $Q= \pi(P)$.  In this situation, the fiber polytope  of $\pi\colon P\to Q$  introduced in \cite{BilleraSturmfels:FiberPolytopes} is the polytope in $V$ defined by 
	$$\Sigma_\pi(P,Q)=  \int_{Q}  P\cap \pi^{-1}(y)\;  dy,$$
	where $dy$ is some positive density on $V/W$.
This integral can be interpreted as  the Minkowski sum of the fibers $P\cap \pi^{-1}(y)$, see \cite[Proposition 1.2]{BilleraSturmfels:FiberPolytopes}.  For every continuous choice $x\in \pi^{-1}(y)$, the integral $$ \int_{Q} f^{-1}(x+P)\,dy,$$
is a translate of $\Sigma_\pi(P,Q)$.  
While this integral may resemble \eqref{eq:def-pullback}, the two are fundamentally different: the latter represents an average of the classes $[f^{-1}(x+P)]$
 taken with respect to the addition in the polytope algebra.

\end{remark}

\begin{remark}\label{rem:pullback}
In the paper \cite{Alesker:Fourier}, Alesker introduced a pushforward of continuous translation-invariant valuations along linear injections. More precisely, if $\Val(V)$ denotes the Banach space  of translation-invariant continuous valuations on convex bodies in $V$, then the pushforward along a linear injection $f\colon W\to V$ is the continuous linear map 
$$ f_* \colon \Val(W)\otimes \Dens(W)^* \to \Val(V)\otimes \Dens(V)^*$$
defined by 
$$ f_* (\phi\otimes \epsilon  )(K)=  \int_{V/f(W)} \phi( f^{-1} (  x+K)) \otimes \epsilon_V\,  d\nu_{V/f(W)}(x),$$
where $K\subseteq V$ is a convex body and  $\epsilon\in \Dens(W)^*$, $\epsilon_V\in \Dens(V)^*$, and $\nu_{V/f(W)}\in \Dens(V/f(W))$ are related via canonical isomorphism of Lemma~\ref{lemma:iso-dens}.
If  $P$ is a polytope  in $V$, then 
$$ \langle f_* (\phi\otimes \epsilon), [P]\otimes \mu \rangle  =    \langle \phi\otimes \epsilon  , f^*([P]\otimes \mu)\rangle.$$
In this sense, the pushforward $f_*$ of translation-invariant continuous valuations can be regarded as dual to the pullback $f^*$ in the polytope algebra. 

\end{remark}

\section{Exterior product}

The exterior product, which we introduce in this section, will play an important role in the construction of the pullback along general linear maps. Moreover, it will facilitate the  evaluation of  the intersection product in certain special situations, see Theorem~\ref{thm:special-elem} and Proposition~\ref{prop:prop-wtV} below.

\begin{theorem}\label{thm:ext} There exists a bilinear map $$\boxtimes \colon \Pi^*(V)\times \Pi^*(W)\to  \Pi^*(V\times W),$$ called the \emph{exterior product}, that is uniquely determined by 
	$$ ([P] \otimes \mu ) \boxtimes  ([Q] \otimes \nu ) = [P\times Q] \otimes (\mu\times \nu).$$
	Moreover, it has the following additional properties:
	\begin{enuma}
	\item It is compatible with the grading: \label{item:ext-grad} $$\Pi^k(V)\boxtimes \Pi^l(W)\subseteq \Pi^{k+l}(V\times W).$$	
	\item \label{item:ext-ass} If $U$ is another vector space and $z\in \Pi^*(U)$, then 
	$$ (x \boxtimes y)\boxtimes z= x \boxtimes (y\boxtimes z).$$
		\item  \label{item:ext-pull} If $f\colon U_1\times U_2\to V\times W$, $f=f_1 \times f_2$, is a linear injection, then
	$$f^*(x\boxtimes y)= f_1^* x \boxtimes f_2^* y.$$
	\item \label{item:ext-e}
	The identity elements for the intersection product satisfy $e_V\boxtimes e_W= e_{V\times W}$.

	 \item \label{item:ext-prod} Let $\diag\colon V\to V\times V$ denote the diagonal embedding. For all $x,y\in \Pi^*(V)$
	 $$  x\cdot y = \diag^* (x\boxtimes y).$$
	 
	\item \label{item:ext-prods} For all $x,y\in \Pi^*(V)$ and $x',y'\in \Pi^*(W)$, 
	$$(x\cdot y)\boxtimes (x'\cdot y')= (x\boxtimes x') \cdot (y\boxtimes y').$$ 
	
	\end{enuma}

\end{theorem}
\begin{proof}
The existence part follows immediately from the universal property of the polytope algebra  and Lemma~\ref{lemma:dilation-cont}. Item \ref{item:ext-grad} follows from $\Delta(\lambda)( x\boxtimes y) = \Delta(\lambda) x \boxtimes \Delta(\lambda) y$, $\lambda>0$, which is straightforward to verify for generators.  Likewise,  \ref{item:ext-ass} is satisfied for generators and hence true in general. By the same token,   \ref{item:ext-pull} follows from 
$$ f^{-1}((v,w) + P\times Q ) = f_1^{-1}( v+ P)\times f_2^{-1}(w+  Q).$$
Property \ref{item:ext-e} concerning the identity elements is also straightforward.

To prove \ref{item:ext-prod} note that 
\begin{align*}  \diag^*\left(([P] \otimes \mu )   \boxtimes  ([P'] \otimes \mu' )\right) &= \int_{V^2/\diag(V)}   [ P\cap (x'-x+ P')]\otimes \mu_V \, d\mu_{V^2/\diag(V)}([(x,x')])\\
& = \int_V   [ P\cap (y+ P')]\otimes \mu_V \, d(g_*\mu_{V^2/\diag(V)})(y)\end{align*} 
where $\mu\times \mu' = \mu_V \otimes  \mu_{V^2/\diag(V)}$ under the isomorphism of Lemma~\ref{lemma:iso-dens} and 
$g\colon V^2/\diag(V)\to V$ denotes the isomorphism $g([(x,x')])=  x'-x$. 

Put $T\colon V^2\to V^2$, $T(x,y)= (x,y-x)$. On the one hand, since $\det T =1$,
\begin{align*}\int_{V} \int_{V^2/\diag(V)} &(h\circ T)(\diag(u)+ (x,x')) d\mu_V(u) d\mu_{V^2/\diag(V)}([(x,x')])\\ &= \int_{V\times V} h(u,v) d(\mu\times \mu')(u,v).\end{align*}
On the other hand, 
 \begin{align*}\int_{V} \int_{V^2/\diag(V)}  & (h\circ T)(\diag(u)+ (x,x')) d\mu_V(u) d\mu_{V^2/\diag(V)}([(x,x')]) \\&= 
 \int_{V\times V} h(u,v) d(\mu_V \times  g_* \mu_{V^2/\diag(V)})(u,v).\end{align*}
We conclude that $\mu\times \mu' = \mu_V \times  g_* \mu_{V^2/\diag(V)}$. 
This shows that \ref{item:ext-prod} holds for generators, which suffices to finish the proof.

Finally, item \ref{item:ext-prods} is a formal consequence of \ref{item:ext-pull} and \ref{item:ext-prod}. Indeed, if $\operatorname{diag}_V$, $\operatorname{diag}_W$, and $\diag_{V\times W}$ denote the diagonal embeddings and $\rho \colon (V\times W)^2\to V^2\times W^2$, $\rho(v_1,w_1,v_2,w_2)= (v_1,v_2,w_1,w_2)$,  then  using $\rho \circ \diag_{V\times W}= \diag_V\times \diag_W$, we obtain
\begin{align*} (x\cdot y)\boxtimes (x'\cdot y') & = \diag_V^* (x\boxtimes y) \boxtimes \diag_W^* (x'\boxtimes y')\\
	& = (\diag_V\times \diag_W)^* (x\boxtimes y \boxtimes x'\boxtimes y')\\
	& = (\diag_{V\times W})^* \rho^*(x\boxtimes y \boxtimes x'\boxtimes y')\\
	& = 	(x\boxtimes x') \cdot (y\boxtimes y')
\end{align*} 
\end{proof}

\begin{corollary} \label{cor:ext-prod}
	Let $m\geq 2$ be an integer, let $x_1,\ldots, x_m\in \Pi^*(V)$, and let $\diag_m\colon V\to V^m$ denote the diagonal embedding. Then
	$$ x_1\cdots x_m= (\diag_m)^*(x_1\boxtimes \cdots \boxtimes x_m).$$
\end{corollary}
\begin{proof}
	For $m=2$, the statement follows directly from Theorem~\ref{thm:ext}\ref{item:ext-prod}. 
The case $m>2$ follows by induction, using $ (\diag_{m-1}\times \id_V)\circ \diag_2= \diag_m$.
\end{proof}

\section{Construction of the pullback in the general case}

\label{sec:pull-gen}

Equipped with the properties of the exterior product, we are now ready to complete the construction of the pullback.

\begin{definition}
Let $f\colon V\to W$ be a linear map. Let $X$ be a complement of $\ker f$ in $V$. Let $e_{\ker f}$ denote the identity element in $\Pi^*(\ker f)$. The pullback along $f$ is defined by 
$$ f^*(x) = e_{\ker f} \boxtimes (f|_X)^*x, \quad x\in \Pi^*(W).$$
\end{definition}
\begin{theorem}Let $f\colon V\to W$ be a linear map. 
	The definition of the  pull-back $f^*$  does not depend on the choice of complement to $\ker f$. Moreover, the pullback has the following additional properties:
	
	\begin{enuma}
		\item It is a morphism of algebras when $ \Pi^*(V)$ and $\Pi^*(W)$ are equipped with the intersection product.
		\item It is compatible with the grading, $f^*(\Pi^k(W))\subseteq \Pi^k(V)$.  	
\item If $g\colon U\to V$ is another linear map, then $(f\circ g)^*= g^* \circ f^*$. 
\item The pullback commutes with the Euler--Verdier involution.
\end{enuma}
\end{theorem}

For the proof, we will use the following lemma.

\begin{lemma}\label{lemma:cartesian}
Let $f\colon V\to W$ be a linear surjection, and let $X$ be a complement to $\ker f$. For all densities  $\mu\in \Dens(\ker f)$ and  $\nu\in \Dens(W)$, 
$$ \rho = \mu  \times   (f|_X)^{-1}_*\nu \in \Dens(V)$$
is independent of the choice of complement $X$.
\end{lemma}
\begin{proof}
Let $X'$ be another complement to $\ker f$, and define $\rho'= 	\mu  \times   (f|_{X'})^{-1}_*\nu$. 
Let $g_1\colon X\to \ker f$ and $g_2\colon 
X\to X'$ be the linear maps, uniquely determined by $v+x= v+g_1(x)+g_2(x)$ for all $v\in \ker f$ and $x\in X$. For all bounded Borel sets $A\subseteq \ker f$ and $B\subseteq X$, one has 
\begin{align*}
\rho'(A\times B) &=  \int_{X'} \int_{\ker f} \mathbf{1}_{A}(v- g_1(g_2^{-1} (x'))) \mathbf{1}_{B}(g_2^{-1}(x')) dv dx' \\
&= \mu(A)\nu( (f|_{X'}\circ g_2)(B))\\
&= \mu(A)\nu( f|_{X}(B))\\
&= \rho(A\times B).
\end{align*} 
This implies $\rho=\rho'$, as claimed. \end{proof}

\begin{proof}  To see that the definition of the pullback is independent of the choice of complement,   define $m=\dim \ker f$ and choose a polytope $Q\subseteq \ker f$ and a density $\mu\in \Dens(\ker f)$ such that $e_{\ker f} = [Q]_m \otimes \mu$. Since $f\colon V\to W$ can be factored as $V\to \im(f)\hookrightarrow W$, we may assume in what follows that $f$ is surjective.
Choose a complement $X$ of $\ker f$ in $V$. 	
For  $x=[P]\otimes \nu\in \Pi^*(W)$, one has
	$$ \nc(e_{\ker f} \boxtimes (f|_X)^* x) =  \sum_{F}  \epsilon_{Q + (f|_{X})^{-1}F} \otimes  (\mathrm{pr}_X)^*  (f|_X)^*N( F,P) \otimes (\mu\times (f|_X)^{-1}_* \nu) ,$$
	where the sum extends over all faces of $P$ and   $\mathrm{pr}_X\colon V\to X$ denotes the projection onto $X$. 	Let $X'$ be another complement to $\ker f$,  and let $\mathrm{pr}_{X'}\colon V\to X'$ denote the projection onto $X'$.
		We claim that each factor in 
		$$\epsilon_{Q + (f|_{X})^{-1}F} \otimes  (\mathrm{pr}_X)^*  (f|_X)^*N( F,P) \otimes (\mu\times (f|_X)^{-1}_* \nu)$$ 
		is independent of the choice of complement $X$ to $\ker f$. Indeed, for the second factor this is a consequence of
		$f\circ \mathrm{pr}_X= f\circ \mathrm{pr}_{X'}$ and for the third one this follows from Lemma~\ref{lemma:cartesian}.

For each face $F$, we will show that $\epsilon_{Q + (f|_{X})^{-1}F}$ and $\epsilon_{Q + (f|_{X'})^{-1}F}$ define the same dual density on $f^{-1}(\bar F)$. Consider  the linear subspaces $L=(f|_{X})^{-1}\bar F$ and $L'= (f|_{X'})^{-1}\bar F$ and the density $\mu$ on 
 $\ker f+ L = f^{-1}(\bar F)= \ker f+L'$ defined by 
 $$  \int_{L}\int_{\ker f } h(v,x) dv dx,$$
 for some positive densities $dv$ and $dx$. Let $g_1\colon L\to \ker f$ and $g_2\colon 
 L\to L'$ be the linear maps uniquely determined by $v+x= v+g_1(x)+g_2(x)$ for all $v\in \ker f$ and $x\in L$. Then, since 
 $f|_X=  f|_{X'} \circ g_2$,  
 \begin{align*} \mu(Q + (f|_{X'})^{-1}F) & = \int_{L}\int_{\ker f }  \mathbf{1}_{Q}(v+g_1(x))\mathbf{1}_{(f|_{X'})^{-1}F} (g_2(x))  dv dx\\
 	&= 
 \mu(Q + (f|_{X})^{-1}F)\end{align*}
 and, in turn, $ \epsilon_{Q + (f|_{X})^{-1}F} =\epsilon_{Q + (f|_{X'})^{-1}F}$.

		We conclude that 
		$$\nc(e_{\ker f} \boxtimes (f|_X)^* x)= \nc(e_{\ker f} \boxtimes (f|_{X'})^* x).$$
		Since $\nc$ is injective, this shows that the pullback does not depend on the choice of complement.

	To prove (a), observe that 
	$$ f^*(x\cdot y) = e_{\ker f} \boxtimes(  (f|_X)^*x \cdot (f|_X)^*y)=  
	(e_{\ker f}  \cdot e_{\ker f} ) \boxtimes(  (f|_X)^*x \cdot (f|_X)^*y)$$
	and use Theorem~\ref{thm:ext}\ref{item:ext-prods}. The equation $f^*(e_W)=e_V$ follows immediately from $(f|_{X})^* e_W= e_X$ and $e_{\ker f} \boxtimes e_X= e_V$, see Theorem~\ref{thm:ext}\ref{item:ext-e}.
	
	(b) is a straightforward consequence of Theorem~\ref{thm:ext}\ref{item:ext-grad} and the corresponding property of the pullback along linear injections.
	
	For the proof of (c) choose a complement $\ker g \oplus X= U$. Put $X_1= \ker(f\circ g|_X)$ and choose complements $X_1\oplus X_2 = X$ and $\ker f \oplus Y = V$ with $Y\supseteq g(X_2)$.
	Since  $g(X_1)\subseteq \ker f$ and $g(X_2)\subseteq Y$ by construction, we may write $g|_X =g_1\times g_2\colon X_1\times X_2 \to \ker f \times Y$. 
	 Using the properties of the pullback along linear injections and  Theorem~\ref{thm:ext}, we compute
	\begin{align*} g^* (f^* (x)) & = e_{\ker g} \boxtimes (g|_X)^* ( e_{\ker f} \boxtimes (f|_Y)^* x)\\
		& = e_{\ker g} \boxtimes (g^*_1( e_{\ker f})  \boxtimes g_2^*(f|_Y)^* x)\\
		& = e_{\ker g} \boxtimes  (e_{X_1}  \boxtimes (f\circ g_2)^* x)\\
			& = (e_{\ker g} \boxtimes  e_{X_1})  \boxtimes (f\circ g_2)^* x\\
& =e_{\ker (f\circ g)}  \boxtimes ((f\circ g)|_{X_2})^* x\\
&  = (f\circ g)^*x. 		
	\end{align*}
	
Finally, (d) follows immediately from (b) and (c).
	
\end{proof}

\begin{remark} \label{rmk:pullback2}
Alesker introduced in \cite{Alesker:Fourier}  the pushforward of continuous translation-invariant valuations along general linear maps. If $f\colon W\to V$ is a linear surjection, then this operation can be described as follows. We keep the notation of Remark~\ref{rem:pullback}.  Choose a complement  $\ker f\oplus X= W$ to the kernel of $f$. Under the canonical isomorphism $\Dens(W)^*\cong \Dens(\ker f)^* \otimes \Dens(X)^*$,  any dual density $\epsilon\in \Dens(W)^*$ is represented as $\epsilon_{\ker f} \otimes  \epsilon_X$. Choose a convex body $E\subseteq \ker f$ so that $\epsilon_{\ker f}(\mu)= \mu(E)$ for all densities $\mu\in \Dens(\ker f)$. The pushforward 
$$ f_* \colon \Val(W)\otimes \Dens(W)^* \to \Val(V)\otimes \Dens(V)^*$$
is defined by 
$$ f_*(\phi\otimes \epsilon )(K) = \frac{1}{m!} \left. \frac{d^m}{dt^m}\right|_{t=0} \phi(t E \times (f|_X)^{-1}(K) )\otimes (f|_X)_*\epsilon_X,$$
where $m=\dim W - \dim V$ and $K$ is a convex body in $V$.
If  $P$ is a polytope  in $V$, then one readily verifies the identity
$$ \langle f_* (\phi\otimes  \epsilon ), [P]\otimes \mu \rangle  =    \langle \phi\otimes \epsilon  , f^*([P]\otimes \mu)\rangle.$$
Hence also in the case of linear surjections,  the pushforward  of continuous translation-invariant valuations can be regarded as dual to the pullback in the polytope algebra.   Since every linear map $f$ can be expressed as $f= j \circ p$, where $j$ is injective and $p$ is surjective, this relationship extends to all linear maps.
\end{remark}

\section{Identities for special elements of the polytope algebra}

In this section, we are concerned with special elements of $\Pi^*(V)$, namely those  proportional to $[P]_{n-k}\otimes \mu$ for $(n-k)$-dimensional polytopes  $P$. Viewed appropriately, the operations of intersection product, pullback, and exterior product of these elements translate into linear algebraic operations on linear  subspaces. 

 Let $L\subseteq V^*$ be a linear subspace. Inclusion and restriction yield an exact sequence
$0 \to L^\perp \to V\to L^*\to 0$. Hence, given  elements of $\Dens(V)$ and $\Dens(L^\perp )^*$, Lemma~\ref{lemma:iso-dens} yields a density on $L^*$.  Combined with \eqref{eq:dual-dens-iso}, this construction describes an isomorphism
\begin{equation}
	\label{eq:dens-perp}   \Dens(L^\perp)^* \otimes \Dens(V) \cong \Dens(L)^*
\end{equation}
that we will frequently use in the following.

\begin{definition}
	Let $L\subseteq V^*$ be a $k$-dimensional linear subspace and let $\epsilon\in \Dens(L)^*$ be a dual density. We denote by $x_{L,\epsilon}$ the unique element in $\Pi^k(V)$ such that $$\nc(x_{L,\epsilon}) = \epsilon\otimes [L]\in \Sigma(V)\otimes \Dens(V)$$ 
	under the isomorphism \eqref{eq:dens-perp}.
\end{definition}

\begin{theorem}\label{thm:special-elem} Let $L\subseteq V^*$  and $L'\subseteq V'^*$ be  linear subspaces  and let $\epsilon\in \Dens(L)^*$ and $\epsilon'\in \Dens(L')^*$ be dual densities. The following statements hold:
	\begin{enuma}
		\item The exterior product satisfies $$x_{L,\epsilon} \boxtimes x_{L', \epsilon'} = x_{L\times L',\epsilon\otimes \epsilon'}.$$
		\item \label{prop:mult flats}
		For any linear map  $f\colon W\to V$,
		$$f^*x_{L,\epsilon} = \begin{cases} x_{f^*L,(f^*)_*\epsilon} & \text{if } f(W)+L^\perp=V,\\
			0 & \text{otherwise}.\end{cases}$$ In the first case,  $f^*L= \{f^* \xi\colon \xi \in L\}$ and
		$f^*\colon L\to f^* L$ is an isomorphism.
	
		\item If $V=V'$, then 
		$$ x_{L,\epsilon} \cdot x_{L',\epsilon'}=\begin{cases}  x_{L+L',a_*(\epsilon\otimes \epsilon')} & \text{if } L\cap L'=\{0\},\\
			0 & \text{otherwise}, \end{cases}$$
		where  $a\colon V^*\times V^*\to V^*$ denotes the vector space addition.
	\end{enuma}

\end{theorem}

For the proof of the theorem,  the following description of the isomorphism \eqref{eq:dens-perp} will be helpful. Recall that $\mu_\omega$ denotes the Liouville measure on $V\times V^*$. If $K\subseteq V$ is a convex body containing the origin in its interior, then $K^\circ \subseteq V^*$ denotes the polar body of $K$.
 
\begin{lemma} \label{lemma:dual-dens} Let $L\subseteq V^*$ be a linear subspace and let $X\subseteq V$ be a complement to $L^\perp$. Suppose the following:
\begin{itemize}
	\item $\nu\in \Dens(V)$ and $\epsilon \in \Dens(L)^*$;
	\item  $P\subseteq L^\perp$ is a convex body with nonempty interior and $\epsilon_P\in \Dens(L^\perp)^*$ denotes the corresponding dual density;
	\item $Q\subseteq X$ is a convex body containing the origin in its interior  such that $\epsilon= \frac{1}{\svol(\pi(Q)^\circ\times \pi(Q))} \epsilon_{\pi(Q)^\circ}$, where $\pi\colon V\to L^*$ is the canonical map.
\end{itemize}	
In this situation, 
	$$ \epsilon_P\otimes \nu \mapsto  \epsilon \quad \text{under the isomorphism \eqref{eq:dens-perp}}$$
if and only if 
$$ \nu(P+Q)=1.$$	
\end{lemma}
\begin{proof}Lemma~\ref{lemma:iso-dens} applied to $0\to L^\perp\to V\to L^*\to 0$ yields $\nu= \nu_{L^\perp} \otimes \nu_{L^*}$ and 
	$$ \nu(P+Q)= \nu_{L^\perp}(P)\nu_{L^*}(\pi(Q)).$$
	Consequently, $\nu(P+Q)=1$ if and only if $\nu_{L^*}(\pi(Q))=1$. By the definition of the isomorphism \eqref{eq:dual-dens-iso}, the latter is equivalent to the statement that the dual density corresponding to $\nu_{L^*}$ is $ \frac{1}{\svol(\pi(Q)^\circ\times \pi(Q))} \epsilon_{\pi(Q)^\circ}$.
\end{proof}	
	
\begin{proof}[Proof of Theorem~\ref{thm:special-elem}]
Choose polytopes $P\subseteq L^\perp$, $P'\subseteq (L')^\perp$  of dimensions $n-k$ and $n-k'$, containing the origin in their relative interior, together with suitable densities $\nu\in \Dens(V)$, $\nu'\in \Dens(V')$ such that 
$	\epsilon_{P}\otimes \nu = \epsilon$ and $	\epsilon_{P'}\otimes \nu' = \epsilon'$ under  the isomorphism \eqref{eq:dens-perp}.
Using that $\lambda^{-(n-k)}[\lambda P] \to [P]_{n-k}$ as $\lambda\to \infty$, we obtain 
\begin{align*} 
	\nc(x_{L,\epsilon} \boxtimes x_{L', \epsilon'})  &=  \lim_{\lambda \to \infty}  \lim_{\lambda' \to \infty}   \lambda^{-(n-k)}
	\lambda'^{-(n-k')}  \nc([\lambda P \times \lambda' P'] \otimes \nu\times \nu')\\
	&= 	\epsilon_{P\times P'} \otimes \nu\times \nu' \otimes [ L\times L']\\
	& = \epsilon \otimes \epsilon' \otimes [ L\times L'].
\end{align*}
 This proves (a).
 
To show (b), let $P\subseteq L^\perp$ be as before  an $(n-k)$-dimensional polytope and $\nu\in \Dens(V)$ a density. 

We first consider the case where $f$ is a linear injection. 
It follows from the definition of the pullback that $f^* x_{L,\epsilon}=0$ if $f(W)+L^\perp \neq V$. We therefore assume from now on that $f(W)+L^\perp =V$. Under this assumption, $f^*\colon L\to f^*L$ is clearly injective. Unraveling the definitions and using Lemma~\ref{lemma:faces-intersection-subspace}, we obtain with $m = \dim W$
\begin{align*} 
	\nc(f^* x_{L,\epsilon})  &=  \lim_{\lambda \to \infty}     \lambda^{-(n-k)}
	\int_{V/f(W)}   \nc([f^{-1}(x+ \lambda P)])\otimes \nu_W\, d\nu_{V/f(W)}([x]) \\
	&= 	 \lim_{\lambda \to \infty}     \lambda^{-(m-k)} 
	\int_{V/f(W)}   \nc([\lambda f^{-1}(y+ P ) ])\otimes \nu_W\, d\nu_{V/f(W)}([y]) \\
	& = \int_{V/f(W)}  \epsilon _{f^{-1}(y+P)} d\nu_{V/f(W)}([y]) \otimes \nu_W \otimes [ f^* L]\\
	&=: \epsilon' \otimes \nu_W \otimes [f^* L].
\end{align*}
Choose a complement $L^\perp \oplus X= V$ such that $X\subseteq f(W)$. This is possible since $f(W)+L^\perp =V$ by assumption.
Choose  a convex body $Q\subseteq X$ containing the origin in its interior such that    $\nu(P+Q)=1$. Note that by   Lemma~\ref{lemma:dual-dens}, 
\begin{equation}\label{eq:dens_Q}
\epsilon= \frac{1}{\svol(\pi(Q)^\circ\times \pi(Q))} \epsilon_{\pi(Q)^\circ}.
\end{equation}
As $(f^*L)^\perp \oplus f^{-1}(X)= W$ and $f^{-1}(y+P)+f^{-1}(Q)= f^{-1}(y+P+Q)$ for $y\in L^\perp$, we have   
\begin{align*}   \langle \epsilon'\otimes \epsilon_{f^{-1}( Q)}, \nu_W\rangle & = \int_{V/f(W)} \nu_W(f^{-1}(y+P) +f^{-1}(Q)) d\nu_{V/f(W)}([y])\\
	& =  \nu(P+Q) =1. 
\end{align*} 
Consequently,  Lemma~\ref{lemma:dual-dens} implies that $\epsilon'\otimes \nu_W$ is mapped to 
\begin{equation}\label{eq:dens_fQ} \frac{1}{\svol( \pi (Q)^\circ \times \pi(Q) )} \epsilon_{f^*(\pi(Q)^\circ) }\end{equation}
	under the isomorphism \eqref{eq:dens-perp}. Comparing this expression with \eqref{eq:dens_Q} shows that \eqref{eq:dens_fQ} equals $(f^*)_* \epsilon$, as claimed.

Assume next that $f\colon W\to V$ is a linear surjection. Choose a complement $X\subseteq W$ to the kernel of $f$ and put $m=\dim \ker f$. Put $P' = (f|_X)^{-1}(P)$ and $\nu'= ((f|_X)^{-1})_* \nu$.  Represent $e_{\ker f}$ as $[R]_m\otimes \rho$. One has
\begin{align*}
\nc(f^* x_{L,\epsilon}) & =  \nc( ([R]_m\otimes \rho) \boxtimes (f|_X)^{*}([P]_{n-k} \otimes \nu) )  \\
& =  \nc( [R \times P'  ]_{m+n-k}\otimes \rho \times \nu' )\\
&=  \epsilon_{R\times P'} \otimes \rho \times \nu' \otimes [f^* L].
\end{align*} 
Choose a complement $L^\perp \oplus Y= V$  and  a convex body $Q\subseteq Y$ containing the origin in its interior such that    $\nu(P+Q)=1$. Put $Q'= (f|_X)^{-1} Q$ and $Y'= (f|_X)^{-1}(Y)$. Notice that 
$ (f^*L)^\perp \oplus Y' = W$ and 
$$ (\rho \times \nu')( R \times P' + Q') =1.$$
If $\pi_W\colon W\to (f^*L)^*$ denotes the canonical map, then Lemma~\ref{lemma:dual-dens} implies that $\epsilon_{R\times P'} \otimes \rho \times \nu'$ is mapped to 
$$ \frac{1}{\svol( \pi_W (Q')^\circ \times \pi_W(Q') )} \epsilon_{\pi_W(Q')^\circ }= (f^*)_*\epsilon$$
	under the isomorphism \eqref{eq:dens-perp}. This finishes the proof of (b), since every linear map $f$ can be expressed as the composition of a surjection and an injection. 

Finally, we prove (c). We will apply the description of the pullback to the diagonal embedding $\diag\colon V\to V\times V$. Note  that $(L\times L')^\perp= L^\perp \times L'^\perp$  and that $\diag(V)+ L^\perp \times L'^\perp = V\times V$ is equivalent to the statement that the image of $L^\perp \times L'^\perp$ under the canonical projection $\pi \colon V\times V\to V\times V/\diag(V)$ equals  $V\times V/\diag(V)$. Observe that $V\times V\to V$, $(x,y)\mapsto x-y$, descends to an isomorphism $T\colon V\times V/\diag(V)\to V$. Since $T(\pi(L^\perp \times L'^\perp)) = L^\perp + L'^\perp$, we conclude that $\diag(V)+ L^\perp \times L'^\perp = V\times V$ is equivalent  to $ V= L^\perp + L'^\perp = (L\cap L')^\perp$.

Let $a\colon V^*\times V^*\to V^*$ denote the vector space addition, $a(\xi,\eta)= \xi+\eta$. It is straightforward to check that $a= \diag^*$. We conclude that 
$ \diag^*(L\times L')= L+L'$ and $(\diag^*)_* = a_*$.  
\end{proof}

If $V= \RR^n$, then the above expression for the product can be made even more explicit. Since there is a canonical positive density on $\RR^n$ and all of its subspaces, namely  the Lebesgue measure $\vol$,  we introduce the following notation.
\begin{definition}
If $V= \RR^n$, then we define $ x_L:= x_{L,\vol}\in \Pi^*(\RR^n)$ for each linear subspace $L\subseteq \RR^n \cong (\RR^n)^*$. Moreover, we abbreviate our notation to $\alpha [P] := [P]\otimes \alpha \vol \in \Pi^*(\RR^n)$ for polytopes $P$ and real numbers $\alpha$. 
\end{definition}

If $L,L'\subseteq \RR^n$ are such that $L\cap L'=\{0\}$, define 
\begin{equation}\label{eq:subspace-det}\sin(L, L')  =  \frac{ \vol(B +B')}{\vol(B) \vol(B')},\end{equation}
where $B,B'$ are the euclidean unit balls in $L$ and $L'$.
In terms of the  principal angles $0\leq \theta_1\leq \cdots\leq  \theta_m\leq \pi/2$, $m=\min(\dim L,\dim L')$,   between $L$ and $L'$, one has
$$ \sin(L,L') = \sin \theta_1\cdots \sin \theta_m,$$
see, e.g., \cite{MiaoBen-Israel:Angles}.
We call $\sin(L,L')$ the sine of the principal angles between $L$ and $L'$. Sometimes this quantity is also called the subspace determinant of $L$ and $L'$, see \cite[Section 14.1]{SchneiderWeil:Stochastic}.

With the above definitions, the  description of the product immediately translates into the following. 
\begin{corollary}\label{cor:mult flats}
	Let $L,L'$ be linear subspaces in $\RR^n$. Then 
	$$ x_L\cdot x_{L'} = 	\begin{cases}  \sin(L,L')\, x_{L+L'} & \text{if } L\cap L'=\{0\},\\
		0  & \text{otherwise}.
		\end{cases}$$
\end{corollary}

There is a particularly simple description of the elements $x_L$ in terms of the pullback along orthogonal projections.

\begin{corollary} Let $L\subseteq \RR^n$ be a linear subspace. Let $p\colon \RR^n \to L$ denote the orthogonal projection. Then $x_L= p^* [\{0\}]$. 
	
\end{corollary}

We return now to the general case of an $n$-dimensional real vector space.
The following proposition shows that multiplication by the element $x_{L,\epsilon}$ is essentially the same as the pullback  along the inclusion $L^\perp \to V$. This fact will be important for us  as it allows us to prove certain statements  by induction on the dimension of $V$.

\begin{proposition}\label{prop:mult-res}
	Let $L\subseteq V^*$ be a line through the origin. Denote $H=L^\perp$ and write $i\colon H\to V$ for the inclusion. Choose a complement $X$ to the hyperplane $H\subseteq V$ and let $\mu_X\in \Dens(X)$ correspond to $\epsilon$ under the isomorphism $X\to L^*$. For every  $y\in \Pi^*(V)$
	$$ x_{L,\epsilon}  \cdot y = ([\{0\}]\otimes \mu_X ) \boxtimes  i^* y.$$
\end{proposition}

\begin{proof}
Let us write $x_{L,\epsilon} = [P]_{n-1} \otimes \nu$ for a suitable $(n-1)$-dimensional polytope $P$ and a density $\nu\in \Dens(V)$. It suffices to prove the statement for elements of the form $y= [Q]\otimes \rho$. In this case, 
we find 
\begin{align*}
	x_{L,\epsilon} \cdot y & = \lim_{\lambda\to \infty} \lambda^{-(n-1)} ([\lambda P]\otimes \nu)\cdot ([Q]\otimes \rho)\\
	& = \lim_{\lambda\to \infty}  \lambda^{-(n-1)} \int_V [ \lambda P \cap (x +  Q)]\otimes \nu \,d\rho(x) \\
	&= \lim_{\lambda\to \infty}  \int_{V/H}  \left( \lambda^{-(n-1)}\int_H  [(x + \lambda P) \cap ( t+Q)] \otimes \nu\, d\rho_{H}(x) \right) d\rho_{V/H}([t]).
\end{align*}
After replacing integration over $H$ by integration over $(-\lambda P)+ (t+Q)$  and the change of variables $\lambda^{-1} x= z$, we find that the inner integral satisfies
\begin{align*}
 \lambda^{-(n-1)}\int_H  &[ (x+\lambda P) \cap (t+ Q)] \otimes \nu \, d\rho_{H}(x)\\ & =
 \lambda^{-(n-1)}\int_{ (-\lambda P)+(t+Q)}  [ (x+\lambda P) \cap (t+ Q)] \otimes \nu \, d\rho_{H}(x)\\ 
 & = \int_{ (- P)+\lambda^{-1} (t+Q)}  [ \lambda(z+ P) \cap (t+ Q)] \otimes \nu \, d\rho_{H}(z).
\end{align*}
For this last integral, it is clear that one can interchange limit  and integral and so
$$\lim_{\lambda\to \infty} \lambda^{-(n-1)}\int_H  [ (x+\lambda P) \cap (t+ Q)] \otimes \nu\, d\rho_{H}(x) =  [ H \cap (t+ Q)] \otimes \rho_H(P) \nu.$$
We conclude that 
$$	x_{L,\epsilon} \cdot y = \int_{V/H}  [ H \cap (t+ Q)] \otimes \rho_H(P) \nu \,d\rho_{V/H}([t]).$$
This last expression resembles the definition of the pullback. The difference is that  here the bracket 
$[ H \cap ( t+Q)]$ is understood  inside $\Pi(V)$ and  not inside $\Pi(H)$. Denoting the latter bracket by $[ H \cap ( t+Q)]_{\Pi(H)}$, we have 
$$  ([\{0\}]\otimes \mu_X ) \boxtimes ([ H \cap ( t+Q)]_{\Pi(H)} \otimes \rho_H) =  [ H \cap ( t+Q)] \otimes (\mu_X \times \rho_H). $$
Thus 
$$([\{0\}]\otimes \mu_X ) \boxtimes  i^* y = \int_{V/H}  [ H \cap ( t+Q)] \otimes  (\mu_X \times \rho_H) \,d\rho_{V/H}([t]).$$
Since 
$ \mu_X \times \rho_H = \rho_H(P) \nu$, this finishes the proof of the proposition.
\end{proof}

\section{An Alexandrov--Fenchel inequality}

The main goal of this section is to prove  Theorem~\ref{mthm:AF}. In fact, we will establish a version of \eqref{eq:HR-inequality} for general convex bodies. 
 The following quantities are a particular instance of a larger family of higher-rank mixed volumes, introduced by Schneider in \cite{schneider2000} under the name of mixed functionals and, independently, by Kotrbat\'y and Wannerer in \cite{KotrbatyWannerer:MixedHR}.
\begin{definition}
	Let $K_1,\ldots, K_{n}$ be convex bodies in $\RR^n$. Let $\iota_i$ denote the inclusion of $\RR^n$ into $(\RR^n)^n = \RR^n \oplus \cdots \oplus \RR^n$ as the $i$-th summand and let $\pi\colon (\RR^n)^n\to \Delta^\perp$ denote the orthogonal projection onto the orthogonal complement of the diagonal $\Delta=\{(x,\ldots,x)\colon x\in\RR^n\}$ in $(\RR^n)^n$. We define
\begin{equation}
	\label{eq:wtV}\wt \V(K_1,\ldots, K_n)=  n^{n/2}\frac{(n(n-1))!}{(n-1)!^n} \V( \pi\circ \iota_1 K_1[n-1],\ldots, \pi\circ\iota_n K_n[n-1])
\end{equation} 
	where the mixed volume on the right-hand side is to be understood inside $\Delta^\perp$ induced by the euclidean structure on $(\RR^n)^n$. 
\end{definition}

Note that the normalization in \eqref{eq:wtV} is different from \cite{KotrbatyWannerer:MixedHR}, but  more convenient for the purposes of the present paper.

We will deduce Theorem~\ref{mthm:AF} from the following Alexandrov--Fenchel-type inequality for the higher-rank mixed volume.

\begin{theorem} \label{thm:AF}
	Let $\mathbf C= (C_1,\ldots, C_{n-2})$ be a tuple of centrally symmetric convex bodies in $\RR^n$. For any convex bodies $K$ and $L$   
	$$\wt \V(K  ,- L ,\mathbf C)^2 \geq  \wt \V(K  ,-K ,\mathbf C)   \wt \V( L,- L ,\mathbf C).$$
	
\end{theorem}

With the added assumption of central symmetry of $L$, Theorem~\ref{thm:AF} was proved in \cite[Theorem 1.4]{KotrbatyWannerer:MixedHR}. The main idea in that paper was to apply  the Fourier transform of smooth translation-invariant valuations  (see \cite{Alesker:Fourier, FaifmanWannerer:Fourier}) to deduce  the desired inequality from the classical Alexandrov--Fenchel inequality.  This approach required the bodies to have a $C^\infty$-smooth and strictly positively curved boundary and $L$ to be centrally symmetric.  It remains to  show 
 that the assumption of central symmetry of $L$  can be omitted. 

Recall from  Section~\ref{sec:convex} that we  denote by $K\# L$  the Blaschke sum of convex bodies in $\RR^n$ with nonempty interior and by $\Proj K$ the projection body of $K$.

\begin{proposition}\label{prop:prop-wtV}
	Let $C_1,\ldots, C_{n}$ be convex bodies in $\RR^n$. The following properties hold:
	\begin{enuma}
		\item If the  bodies $C_1,\ldots, C_{n}$ are polytopes, then  $$ \ell_{C_1} \cdots \ell_{C_n} =   \wt \V(C_1,\ldots, C_n).$$
		\item $\wt \V$ is a continuous function on $\calK(\RR^n)\times \cdots \times  \calK(\RR^n)$. 
		\item \label{item:B-sum} For all convex bodies $K,L\subseteq \RR^n$ with nonempty interior,
		$$ \wt \V(K\# L , C_1,\ldots, C_{n-1}) = \wt \V(K , C_1,\ldots, C_{n-1}) + \wt \V(L , C_1,\ldots, C_{n-1}).$$
		
		\item If the bodies $C_1,\ldots, C_{n-1}$ are centrally symmetric, then 
		$$ \frac{2^n}{n!} \wt \V(C_1,\ldots, C_n) =  \V(\Proj C_1,\ldots, \Proj C_n).$$
		\item If the bodies $C_1,\ldots, C_{n}$ are full-dimensional and centrally symmetric, then 
	$ \ell_{C_1} \cdots \ell_{C_n}>0$.
	\end{enuma}
\end{proposition} 
With a different proof, item (d) has already appeared in  \cite[Proposition 5.6]{KotrbatyWannerer:MixedHR}.
For the proof of Proposition~\ref{prop:prop-wtV}, we need the following simple lemma. 
\begin{lemma} \label{lemma:ell}Let $P,P'\subseteq \RR^n$ be polytopes with nonempty interior. Then 
	$$ \ell_P+ \ell_{P'}= \ell_{P\# P'}.$$	 
\end{lemma}
\begin{proof}
On the one hand, the area measure of a polytope is the discrete measure
$$ S_P= \sum_{i=1}^m \vol_{n-1}(F_i) \delta_{u_i}$$
where $F_1,\ldots, F_m$ are the facets of $P$ with facet normals $u_1,\ldots, u_m$. 
On the other hand, applying the normal cycle embedding yields
$$ \nc (\ell_P) =  \sum_{i=1}^{m} \vol_{n-1}(F_i) [N(F_i,P)].$$
Since by definition $S_{P\# P'}= S_P+ S_{P'}$, the claim follows  from the injectivity of $\nc$. 
\end{proof}

\begin{proof}[Proof of Proposition~\ref{prop:prop-wtV}]
To prove (a) let $\chi\colon \Pi^n(\RR^n) \to \RR $ denote the Euler characteristic, $\diag_n\colon \RR^n\to (\RR^n)^n$ the diagonal embedding, and $\iota_i$ the inclusion of $\RR^n$ into $(\RR^n)^n$. Let $\Delta= \diag_n(\RR^n)$ denote the diagonal. Identifying $(\RR^n)^n/\Delta = \Delta^\perp$, Lemma~\ref{lemma:iso-dens} shows that $d\vol_{(\RR^n)^n/\Delta} = n^{n/2} dx$, where $dx$ is the Lebesgue measure induced from the euclidean structure on $(\RR^n)^n$.  By Corollary~\ref{cor:ext-prod}, for any positive numbers $\lambda_1,\ldots, \lambda_n$ we have 
\begin{align*} 
\chi ( [\lambda_1 P_1] \cdots [\lambda_n P_n] )  & =  \chi ( \diag_n^* ([\lambda_1 \iota_1 P_1 + \cdots + \lambda_n \iota_n P_n]))\\
& = n^{n/2}\int_{\Delta^\perp} \chi( [ (\lambda_1 \iota_1 P_1 + \cdots + \lambda_n \iota_n P_n)\cap (x+\Delta)])  dx\\
&= n^{n/2} \vol ( \lambda_1 \pi\circ  \iota_1 P_1 + \cdots+ \lambda_n \pi \circ \iota_n P_n).
\end{align*} 
Expanding  both the first and the last expression into a polynomial in the $\lambda_i$ and comparing the coefficients of the 
monomial $\lambda_1^{n-1}\cdots \lambda_n^{n-1}$, the claim follows.

(b) is an immediate consequence of the definition of $\wt \V$ and the continuity of the mixed volume.

Since $\wt \V$  and, by Lemma~\ref{lemma:cont-B}, Blaschke addition  are continuous in the Hausdorff metric, (c) follows  via approximation by polytopes from (a) and Lemma~\ref{lemma:ell}.

To prove (d) assume first that the bodies $C_1,\ldots, C_n$ are centrally symmetric polytopes. If $C$ is a centrally symmetric polytope, then  
$$ \ell_{C} = \sum_{i=1}^m \vol_{n-1}(F_i) x_{\RR u_i}.$$ 
Therefore, using Corollary~\ref{cor:mult flats} we obtain 
$$ \ell_{C_1} \cdots \ell_{C_n} =  \sum_{i_1=1}^{m_1} \cdots \sum_{i_n=1}^{m_n} |\det(u_{i_1},\ldots, u_{i_n})| \vol_{n-1}(F_{i_1}) \cdots   \vol_{n-1}(F_{i_n}).
$$ 
For centrally symmetric bodies $C$, equation~\eqref{eq:projection-body} implies 
$$ \Proj C= \sum_{i=1}^m
 \vol_{n-1}(F_i)[-u_i, u_i], $$
and hence  $ \V(\Proj C_1,\ldots, \Proj C_n)$ equals
$$ 
  \sum_{i_1=1}^{m_1} \cdots \sum_{i_n=1}^{m_n} \V([u_{i_1},-u_{i_1}], \ldots, [u_{i_n},-u_{i_n}]) \vol_{n-1}(F_{i_1}) \cdots   \vol_{n-1}(F_{i_n})
$$
Since $ \V([u_{i_1},-u_{i_1}], \ldots, [u_{i_n},-u_{i_n}]) = \frac{2^n}{n!}  |\det(u_{i_1},\ldots, u_{i_n})|$, this proves (d) for centrally symmetric polytopes and by approximation for all centrally symmetric convex bodies. It suffices to prove the general case for convex bodies with nonempty interior.  If all bodies except perhaps the last one are centrally symmetric, then 
\begin{align*} 2 \wt \V(C_1,\ldots, C_n) & =  \wt \V( C_1,\ldots, C_{n-1}, C_n \# (-C_n))\\
	&=\frac{n!}{2^{n}} \V(\Proj C_1,\ldots, \Proj C_{n-1}, \Proj (C_n \# (-C_n))). \end{align*} 
Since $\Proj (K \# L) = \Proj K + \Proj L$, the claim follows.

Finally, (e) follows from (a) and (d).
\end{proof}

\begin{lemma}\label{lemma:existence-signed}
For every finite signed Borel measure  $\mu$ on $S^{n-1}$ with centroid at the origin,
 there exist convex bodies $K$ and $L$ such that 
$$ \mu = S_K- S_L.$$
Moreover, if $\mu$ is discrete, then there exist polytopes $K$ and $L$ with this property.
\end{lemma}
\begin{proof}
There exist nonnegative finite Borel measures such that $\mu=\mu_+ - \mu_-$ and both measures have their centroids at the origin. Adding a measure to both if necessary, we can assume that $\mu_+$ and $\mu_-$ are not concentrated on an equator. Thus $\mu_+$ and $\mu_-$ satisfy the hypothesis of Minkowski's existence theorem and hence there exist convex bodies $K,L$ as claimed.
\end{proof}

\begin{definition}
Let $C_1,\ldots, C_{n-1}$ be convex bodies in $\RR^n$ and let $\mu$ and $\nu$ be  finite signed  Borel measures  on the unit sphere with centroid at the origin. We define 
$$ \wt \V(\mu, C_1,\ldots, C_{n-1})= \wt \V(K, C_1,\ldots, C_{n-1})- \wt \V(L, C_1,\ldots, C_{n-1}),$$
where $K,L$ any convex bodies satisfying  $\mu= S_K - S_L$. The expression $ \wt \V(\mu,\nu, C_1,\ldots, C_{n-2})$ is defined analogously.
\end{definition}

\begin{lemma} The definition of  $\wt \V(\mu, C_1,\ldots, C_{n-1})$ is independent of the choice of decomposition  $\mu= S_K - S_L$. An analogous statement holds for  $\wt \V(\mu,\nu, C_1,\ldots, C_{n-2})$.  
\end{lemma}
\begin{proof}
We present a short proof based on McMullen's characterization \cite{McMullen:Continuous} of  $(n-1)$-homogeneous continuous  translation-invariant valuations  on convex bodies in $\RR^n$. 
 Since the function $\phi(K)= \wt \V(K,C_1,\ldots, C_{n-1})$ is such a valuation, there exists  by  McMullen's theorem  a  continuous function $f\colon S^{n-1}\to \RR$ on the unit sphere such that for every  convex body $K$
\begin{equation}\label{eq:McMullen-char} \phi(K)= \int_{S^{n-1}} f(u) dS_K(u).\end{equation}

Suppose that $\mu= S_K-S_L=S_{K'}-S_{L'}$. To prove the lemma, it suffices to show that 
$$ \phi(K)-\phi(L)= \phi(K')-\phi(L').$$
This follows immediately from \eqref{eq:McMullen-char}.
\end{proof}

The classical Alexandrov--Fenchel inequality has three equivalent formulations, see \cite[Lemma 3.11]{ShenfeldHandel:Extremals}. Analogous equivalent formulations exist in the context of  Theorem~\ref{thm:AF}.
Let  $a(u)=-u$, $u\in S^{n-1}$, denote the antipodal map.

\begin{lemma}\label{lemma:equivalent-forms}
	Let $\mathbf C= (C_1,\ldots, C_{n-2})$ be a tuple of centrally symmetric convex bodies in $\RR^n$. Then the following are equivalent: 
	\begin{enuma}
		\item   For all convex bodies $K$ and $L$ 
		$$\wt \V(K  ,- L ,\mathbf C)^2 \geq  \wt \V(K  ,-K ,\mathbf C)   \wt \V( L,- L ,\mathbf C).$$
		\item \label{item:equiv-2} For all finite signed  Borel measures $\mu$ with centroid at the origin  and all convex bodies $L$ 
		$$\wt \V( \mu ,- L ,\mathbf C)^2 \geq  \wt \V(\mu  ,a_* \mu ,\mathbf C)   \wt \V( L,- L ,\mathbf C) .$$
		\item \label{item:equiv-3}For all  finite signed  Borel measures $\mu$ with centroid at the origin and all convex bodies $L$ with $\wt \V(L,-L,\mathbf C)>0$, 
		$$ \wt \V(\mu, -L,\mathbf C)=0 \text{ implies }  \wt \V(\mu, a_* \mu,\mathbf C)\leq 0.$$
	\end{enuma}
	Moreover, if $L$ is throughout assumed to be centrally symmetric, then these statements are also equivalent. 
\end{lemma}
\begin{proof} The implications   (b)$\Rightarrow$(c)  and (b)$\Rightarrow$(a) are trivial. 
	Assume that (a) holds.
	Suppose $\mu = S_{K'} - S_{K''}$. Since $\wt \V$ is continuous by Proposition~\ref{prop:prop-wtV}, approximating $K'$,  $K''$, and $L$ by convex bodies with smooth and strictly positively curved boundary,  we may assume that $\mu$ and $S_L$  have  continuous densities with respect to the spherical Lebesgue measure, $d\mu= f(u) du$ and $dS_L= f_L(u) du$, and that   $f_L$ is strictly positive. Hence there exists a number $\alpha >0$ such that 
	$(f + \alpha f_L) - \alpha f_L$ is a decomposition into strictly positive functions. Since $$\int_{S^{n-1}} u(f + \alpha f_L)(u) du=0,$$ the hypothesis of Minkowski's existence theorem is satisfied and, consequently,  there exists a convex body $K$ such that 
	$\mu= S_K - \alpha S_L$.  Plugging this decomposition into (b) and expanding both sides, we see that the terms containing $\alpha$ cancel. This 
	shows (a)$\Rightarrow$(b).
	
	Suppose that (c) holds. To show (b) we may assume by approximation that  $\wt\V(L,-L,\mathbf C)>0$. Since $$
	\wt \V(\mu- \alpha S_L,-L,\mathbf C)=0$$ for $\alpha = \wt \V(\mu, -L,\mathbf C)/\wt \V(L,-L,\mathbf C)$, (c) yields
	$$ 0 \geq \wt \V(\mu- \alpha S_L,a_*(\mu-\alpha S_L),\mathbf C)= \wt \V(\mu,a_*\mu, \mathbf C) - \frac{\wt \V(\mu, -L,\mathbf C)^2}{\wt \V(L,-L,\mathbf C)}.$$
	This finishes the proof of (c)$\Rightarrow$(b).
	
	Finally, the above proof works without change if one assumes that $L$ is centrally symmetric.
\end{proof}

\begin{proof}[Proof of Theorem~\ref{thm:AF}] For centrally symmetric convex bodies  $L$, the statement was proved  in \cite[Theorem 1.4]{KotrbatyWannerer:MixedHR}. Therefore, all the  statements of Lemma~\ref{lemma:equivalent-forms} hold for centrally symmetric convex bodies $L$.  
	
 Let   $L$  be a convex body with nonempty interior, possibly not centrally symmetric, and suppose $\wt \V(L,-L,\mathbf C)>0$. Put $L':= L\#(-L)$. Then $L'$ is centrally symmetric and $$ \wt  \V(L,L',\mathbf C)= \wt \V(L,L,\mathbf C) + \wt \V(L,-L,\mathbf C)>0.$$ 
	
	Suppose $\wt \V(\mu, -L, \mathbf C)=0$. 	Since   $\wt \V(\mu- \alpha S_L, L',\mathbf C) =0$ for $$\alpha =   \wt \V(\mu, L',\mathbf C)/ \wt \V(L, L',\mathbf C)$$
and since $L'$ is centrally symmetric,
$$ 0 \geq \wt \V(\mu- \alpha S_L, a_*(\mu- \alpha S_L),\mathbf C)= \wt \V(\mu,a_*\mu, \mathbf C) + \alpha^2 \wt \V(L,-L,\mathbf C). $$
	We conclude that  $0\geq \wt \V(\mu,a_*\mu, \mathbf C)$. Thus we proved that  Lemma~\ref{lemma:equivalent-forms}\ref{item:equiv-3} holds for convex bodies $L$ with nonempty interior. The proof of (c)$\Rightarrow$(b) applies without change and shows that  Lemma~\ref{lemma:equivalent-forms}\ref{item:equiv-2} holds for convex bodies $L$ with non-empty interior. A standard approximation argument now concludes the proof.
\end{proof}

\begin{proof}[Proof of Theorem~\ref{mthm:AF}] Choose a euclidean inner product to identify $V$ with $\RR^n$ such that $\vol$ is the Lebesgue measure. Let $x\in \Pi^1(\RR^n)$ be given. 
	By Lemma~\ref{lemma:existence-signed}, we may write $x= \ell_P - \ell_{P'}$ for certain polytopes $P$ and $P'$.  Hence
	$$ -\sigma(x)\cdot \ell_Q \cdot  \ell_{\mathbf C}  = (\ell_P - \ell_{P'}) \cdot \ell_{-Q} \cdot  \ell_{\mathbf C} = \wt \V(\mu, -Q, \mathbf C)$$
	with $\mu= S_P- S_{P'}$.  The assertion follows now  immediately from Theorem~\ref{thm:AF} in the formulation of  Lemma~\ref{lemma:equivalent-forms}\ref{item:equiv-2}. 
\end{proof}

\section{Finite-dimensional subalgebras}
\label{sec:finite-dim}

In this section, we take a closer look at the subalgebras $\A^*(E)$ and $\A_+^*(E)$, defined in the introduction, and their connection with algebraic combinatorics.
Throughout this section, we work with a fixed positive density $\vol$ on $V$.  Let $E$ be a finite set of lines in $V^*$  that pass through the origin and are not contained in a single hyperplane.

First of all, observe  that $$\nc(\A^1(E))\subseteq \linspan \{[L^+],[L^-]\colon L\in E\},$$ where $L^+$ and $L^-$ denote the two half lines corresponding to $L$. Hence $\A^1(E)$ and, consequently,  $\A^*(E)$ are indeed finite-dimensional.

Recall that $\sin(L,L')$ denotes the sine of the principal angles between $L$ and $L'$. 
\begin{proposition}\label{prop:dim A}	Let $E$ be a finite set of lines in $\RR^n$ that pass through the origin and are not contained in a single hyperplane. As a vector space, 
	$ \A_+^*(E)$ is spanned by the linearly independent elements $\{ x_L\colon L\in \calL_k(E)\}$. The algebra structure is determined by 
	\begin{equation}\label{eq:algA} x_L\cdot x_{L'} = 	\begin{cases}  \sin(L,L')\, x_{L+L'} & \text{if } L\cap L'=\{0\},\\
		0  & \text{otherwise}.
	\end{cases}\end{equation}
Moreover, $$\Kc(E)= \left\{ \sum_{L\in E} c_L x_L\colon c_L>0 \text{ for all }L\in E\right\}.$$
\end{proposition}
\begin{proof} If $\ell_P$ belongs to $\A_+^1(E)$, then   $\ell_P = \sum_{L\in E} c_L x_L$, where  $c_L$ is  the $(n-1)$-dimensional volume of the facet of $P$ perpendicular to $L$. Applying Corollary~\ref{cor:mult flats}, one obtains  $$\A_+^k(E)= \linspan\{ x_L\colon L \in \calL_k(E)\}$$ and 
	the description of the algebra structure. By the injectivity of the normal cycle embedding,  the elements $x_L$, $L \in \calL_k(E)$, are linearly independent.

The description of $\Kc(E)$ is an immediate consequence of Minkowski's existence theorem.
\end{proof} 

Note that the proposition implies 	
$$\dim \A^k_+(E) = |\calL_k(E)|.$$
Consequently, the injective hard Lefschetz property of Conjecture~\ref{conj:Kahler} directly implies the  Dowling--Wilson conjecture.

 In general,  one expects  for $k\leq n/2$ $$ \dim \A^k(E)<\dim \A^{n-k}(E).$$ It follows that  Poincar\'e duality, while valid in $\Pi^*(V)$, will  not hold in this setting. 
The M\"obius algebra, which appears in the context of the Dowling--Wilson conjecture and which we will define next, displays an analogous lack of  Poincar\'e duality. 
\begin{definition}
	Let $E$ be a finite set of lines in $\RR^n$  that pass through the origin and are not contained in a single hyperplane.
	 The graded M\"obius algebra is the free  real vector space $ \Moe^*(E)$
	spanned by the elements $y_L$, $L\in \bigcup_{k=0}^n \calL_k(E)$ together with the product defined by
	\begin{equation}\label{eq:algM} y_L\cdot y_{L'} = \begin{cases} y_{L+L'} & \text{if } L\cap L'= \{0\},\\
		0 & \text{otherwise.}
	\end{cases}\end{equation}
	The grading compatible with the product is defined by $\Moe^k(E)= \linspan\{ y_L\colon L\in \calL_k(E)\}$.
\end{definition}

The comparison of \eqref{eq:algM} with \eqref{eq:algA} suggests that  $\A_+^*(E)$ should be regarded as a volumetric version of the graded M\"obius algebra. The  injective hard Lefschetz property and the Hodge--Riemann relations within the context of the graded  M\"obius algebra were recently established in  \cite{HuhWang:Enumeration,BHMPW:Singular}. 

We close this section with a description of $\A^*(E)$ in the simplest case.

\begin{example}If $E\subseteq V^*$ is as  in Definition~\ref{def:A}, then $E$ consists of at least $n$ lines. Suppose that  $|E|=n$. Choose a euclidean inner product to identify $V\cong \RR^n$ so that $E$ consists of the coordinate axes and $\vol$ is the Lebesgue measure. Then, $\calL_k(E)$ is the set of all $k$-dimensional coordinate subspaces of $\RR^n$. If $\ell_P\in \A^1(E)$, then $P$ must be a box $P=[a_1,b_1]\times \cdots \times [a_n,b_n]\subseteq \RR^n$ and therefore centrally symmetric. It follows that $\A^*(E)= \A^*_+(E)\cong\Moe^*(E)$. Consequently, 
Conjecture~\ref{conj:Kahler}  in the case $\ell_{C_{0}}=\cdots =\ell_{C_{n-2k}}$ is implied by the corresponding statement for the graded M\"obius algebra.
\end{example}

\section{The degree one case}

The goal of this section is to prove the following.  

\begin{theorem}\label{thm:equality}
Let     $\ell_{C_1},\ldots, \ell_{C_{n-2}}\in \Kc(E)$ and define $\ell_{\mathbf C}= \ell_{C_1}\cdots \ell_{C_{n-2}}$.	Let $Q$ be a polytope in $V$ satisfying  $\sigma(\ell_{Q}) \cdot \ell_{Q} \cdot \ell_{\mathbf C} \neq 0$. Then, for all $x\in \A^1(E)$,
\begin{equation}\label{eq:equality} \sigma(x) \cdot \ell_{Q}\cdot \ell_{\mathbf C} =0 \quad \text{and} \quad \sigma(x)\cdot x \cdot  \ell_{\mathbf C}=0\end{equation}
implies $x=0$.	
\end{theorem}

 Before we move to the proof of Theorem~\ref{thm:equality}, let us first explain how it implies Theorem~\ref{mthm:1}.

\begin{proof}[Proof of Theorem~\ref{mthm:1}]
(a)  To prove the injectivity of the Lefschetz map suppose that $x\in \A^1(E)$ satisfies 
$ x\cdot \ell_{\mathbf{C}} =0$. As an immediate consequence we have
\[ \sigma(x) \cdot \ell_{C_1} \cdot  \ell_{\mathbf C}=(-1)^{n-1}\sigma( x \cdot \ell_{C_1} \cdot \ell_{\mathbf C} ) =0\]
and $\sigma(x)\cdot x \cdot  \ell_{\mathbf C}=0$. Since $\sigma(\ell_{C_1}) \cdot \ell_{C_1} \cdot \ell_{\mathbf C}\neq 0$ by Proposition~\ref{prop:prop-wtV}(e),  we can apply Theorem~\ref{thm:equality} to obtain $x=0$.

(b) follows from Theorem~\ref{mthm:AF}, by a routine argument:
If $x \cdot  \ell_{C_0} \cdot \ell_{\mathbf C} =0$, then Theorem~\ref{mthm:AF} implies
\[ 0   = \big( \sigma(x) \cdot  \ell_{C_0} \cdot \ell_{\mathbf C}\big)^2 \geq \big(
\sigma(x) \cdot  x \cdot \ell_{\mathbf C}\big) \big(
\sigma(\ell_{C_0} ) \cdot  \ell_{C_0} \cdot \ell_{\mathbf C}\big) 
.\] 
Since $\sigma(\ell_{C_0} ) \cdot  \ell_{C_0} \cdot \ell_{\mathbf C}<0$, we arrive at the desired inequality. If equality occurs in this inequality, then applying Theorem~\ref{thm:equality} with $Q=C_0$ yields $x=0$. 
\end{proof}

The overall strategy for the proof of Theorem~\ref{thm:equality} is to  give a direct argument for $n=2$ and to treat the case  $n>2$ by induction.

\subsection{The $2$-dimensional case}

It is a classical fact that if $L$ is a convex body in the plane with $\V(L,L)>0$,  then for any difference of support functions $f$, the equations  
$$ \V(f,L)= 0\quad  \text{and}\quad  \V(f,f)=0$$
imply that $f$ is a linear functional. See \cite[Lemma 3.12]{ShenfeldHandel:Extremals} for a simple and direct proof.  Using this fact, we readily verify Theorem~\ref{thm:equality} in dimension two. Lemma~\ref{lemma:HLHR} below provides another proof.

\begin{lemma}\label{lemma:dim2}
Theorem~\ref{thm:equality} holds for $n=2$.
\end{lemma}
\begin{proof}
We know from   Lemma~\ref{lemma:pairing} that 
	$$ \sigma(\ell_{P_1}) \cdot \ell_{P_2} = -2\V(P_1,P_2)$$ 
	for all polytopes $P_1$ and $P_2$. 
	Suppose that  $x= \ell_P-\ell_{P'}$ for certain polytopes $P$ and $P'$,   and define $f= h_P - h_{P'}$. Assume that 
	$$0= \sigma(x)\cdot \ell_{Q}  = -2\V(f,Q) \quad \text{and} \quad  0=\sigma(x)\cdot x = - 2\V(f,f).$$
	Since $2\V(-Q,Q)= \ell_{Q}\cdot \ell_{Q}>0$ is equivalent to $\V(Q,Q)>0$,  these  equalities, as just mentioned,  imply that $f$ is a linear functional. We conclude that $x=0$. 
\end{proof}

\subsection{Restrictions}

The goal of this subsection is to prove that every $x\in \A^1(E)$ is determined by its restrictions to hyperplanes perpendicular to the lines in $E$. This property will be a key ingredient of our inductive setup to characterize when equality occurs in the Hodge--Riemann relations.

\begin{lemma}\label{lemma:nc-res}
	Let $H\subseteq V$ be  a linear hyperplane and let $j\colon H\to V$ denote the inclusion. The following statements hold for all $x\in \A^1(E)$:
	\begin{enuma}
			
\item  $j^* x\in \A^1(j^*E)$, where $j^*E = \{ j^*L \colon L\in E\} 
	\setminus\{ 0\}$. 
	\item More precisely, if $\nc(x) = \sum_{i=1}^m {\alpha_i} [N_i]$ with nonzero coefficients $\alpha_i$,  then $$ \nc(j^*x)= \sum_{i\colon j^* N_i\neq \{0\}} \alpha'_i [j^*N_i]$$ 
	with nonzero coefficients $\alpha_i'$. 
	\item If $x\in \Kc(E)$, then $j^*x\in \Kc(j^*E)$. 
	\end{enuma} 
	
\end{lemma}
\begin{proof}
	It suffices to prove (a) and (b) for $x= \ell_P$.  	
Choose a euclidean inner product to identify $V$ with $\RR^n$ such that $\vol$ is the Lebesgue measure.  
	 Let $F_1,\ldots, F_m$ and $v_1,\ldots, v_m$ denote the facets and the corresponding facet normals of $P$.
	 
	Let $\pi\colon \RR^n\to H$ denote the orthogonal projection and let $\RR_+=\{ \alpha\in \RR \colon \alpha>0\}$. 
	Let   $u$ be a unit normal vector to $H$. One has
	\begin{align*}  \nc(j^* x) &= 
		  \int_{\RR} \nc([ P\cap (H+tu) ]_{n-2})  dt \\
		 & = \sum_{i=1}^m  \int_{\RR}  \vol_{n-2}(F_i\cap (H+tu)) dt \, [ \RR_+ \pi(v_i)]\\
		  & = \sum_{i=1}^m  |\pi(v_i)| \vol_{n-1}(F_i) [ \RR_+ \pi(v_i)].
		 \end{align*} 
	This proves (b).
The discrete measure on $S^{n-2}$
$$ \mu = \sum_{i\colon \pi(v_i)\neq 0} |\pi(v_i)| \vol_{n-1}(F_i) \delta_{ \pi(v_i)/|\pi(v_i)|}$$
has centroid at the origin and is not concentrated on an equator. Consequently, by Minkowski's existence theorem, there exists a polytope $P'\subseteq H$ such that $\mu= S_{P'}$. We conclude that 
$j^* x= \ell_{P'}\in \A^1(j^*E)$. This finishes the proof of (a). If $P$ is centrally symmetric, then so is $P'$ and (c) follows.
\end{proof}

To prove the next proposition, we need the following lemma due to Motzkin \cite{Motzkin:Lines}.
\begin{lemma}\label{lemma:projective}
Given finitely many points in the real projective plane, not all collinear,  there exists a line that contains exactly two of the points.
\end{lemma}

\begin{proposition}\label{prop:res}

	Let $x\in \A^1(E)$. Suppose that $j^*_{L} x =0$ for all $L\in E$, where $j_{L}\colon L^\perp\to V$ denotes the inclusion. If $n\geq 3$, then $x=0$. 
\end{proposition}

\begin{remark} 
	\begin{enuma}
\item 		
The analogous statement for $n=2$ is false. Indeed, since $L^\perp$ is one-dimensional for $n=2$, it follows that $\sigma(j_L^*x)=-j^*_Lx$  for all $x\in \Pi^1(V)$. Therefore, any $x\in \A^1(E)$ with $\sigma(x)= x$ satisfies $j^*_Lx=0$. 
\item For general $x\in \Pi^1(V)$, the condition $j^*_{L} x =0$ for all $L\in E$ does not imply $x=0$. 
\end{enuma}
\end{remark}

\begin{proof}[Proof of Proposition~\ref{prop:res}]We will prove the claim by induction on $n$. First, suppose $n=3$ and $V=\RR^3$. We argue by induction on $|E|\geq 3$.  If $|E|=3$, then the claim is easily seen to be true. Indeed, if $E=\{L_1,L_2,L_3\}$, then, in view of Lemma~\ref{lemma:nc-res}, the  condition $j^*_{L_1} x=0$ forces $\nc(x)= \alpha_- [L_1^-]+ \alpha_+ [L^+_1]$. By the same reasoning, applied to $L_2$ (or $L_3$) we conclude that $x=0$. 
		
Suppose now that $|E|>3$ and that the claim holds for sets $E$ of smaller cardinality. By Lemma~\ref{lemma:projective} there exists a linear  hyperplane $H\subseteq \RR^3$ that contains exactly two lines from $E$, say $L$ and $L'$. Consider the orthogonal projection $\pi_L\colon \RR^3 \to L^\perp$. Since $H$ contains exactly two lines from $E$, if $L''\in E$ is a line $\neq L,L'$, then $\pi_{L} L' \neq \pi_L L''$. Hence, if $ \nc(x)= \sum_{M\in E} \alpha_{M^\pm} [M^\pm]$, then by Lemma~\ref{lemma:nc-res} the condition $j_L^* x=0$  implies  $\alpha^\pm_{L'} =0$.  We  conclude that $x\in \A^1(E\setminus\{L'\})$ and apply the inductive hypothesis.  This finishes the proof in dimension three. 
	
	Now, assume $n>3$ and that the proposition holds in dimension $n-1$. 	
	Let $H\subseteq V$ be a linear hyperplane and let $i_H\colon H\to V$ denote the inclusion. Suppose that $H^\perp \notin E$.
	Then $H\cap L^\perp$ has dimension $n-2$ for every $L\in E$. Define $E'= i_H^* E$ and $x'= i_H^* x$. Then $E'$ is a set of lines in $H^*$ and $x'\in \A^1(E')$.  For $L'= i^*_H L\in E'$ let $k_L\colon H\cap L^\perp \to L^\perp$ and $j_{L'}\colon (L')^\perp \to H$ denote the inclusions. Since $H\cap L^\perp = (L')^\perp$,  $i_{H} \circ j_{L'} = j_L\circ k_L$ and thus 
	$$ (j_{L'})^* x' = k_L^* (j_L^* x)= 0$$
	for any $L'\in E'$.  By the inductive assumption $x'=0$.  Thus we conclude that $i_H^*x=0$ for all linear hyperplanes in $V$.
	
	Let $H$ be a hyperplane in $V$ and let $L$, $L'$ be two distinct lines in $V^*$. Observe that if  $i^*_H L= i^*_H L'$ then  $H^\perp\subseteq L + L'$. Hence if $H$ is chosen so that $H^\perp$ is not contained in the finitely many planes $L+L'$ for distinct  $L, L'\in E$, then $i^*_H L\neq  i^*_H L'$ whenever $L\neq L'$. 	
With the help of Lemma~\ref{lemma:nc-res}, we can now deduce   $x=0$  from $i_H^* x=0$ for a generic linear hyperplane $H$.	
\end{proof}

\subsection{The higher dimensional case} 

We start by relating the kernel of the Lefschetz map and the radical of the Hodge--Riemann form. 

\begin{lemma}  \label{lemma:HLHR}Let    $C_1,\ldots, C_{n-2}$ be centrally symmetric polytopes  in $V$, and let $Q$ be a polytope in $V$ satisfying 
	$ \sigma(\ell_Q) \cdot \ell_Q \cdot \ell_{\mathbf C}\neq 0$. Then, for every $x\in \Pi^1(V)$, the following are equivalent:
	\begin{enuma} 
		\item $\sigma(x)\cdot \ell_Q\cdot \ell_{\mathbf C} =0$ and $\sigma(x)\cdot x \cdot  \ell_{\mathbf C} =0$.
		\item  $x\cdot \ell_\mathbf{C}=0$. 
	\end{enuma}
\end{lemma}
\begin{proof} The implication (b)$\Rightarrow$ (a) is trivial  (see the proof of Theorem~\ref{mthm:1}).
	
	Since $\dim \Pi^n(V)=1$ and since  $ \sigma(\ell_Q) \cdot \ell_Q \cdot \ell_{\mathbf C}\neq 0$ by  hypothesis, it follows that the condition $\sigma(x)\cdot \ell_{Q} \cdot \ell_{C_1} \cdots \ell_{C_{n-2}}=0$ defines a hyperplane $H$ in $\Pi^1(V)$ and that one has the decomposition
	\begin{equation} \label{eq:decomp} \Pi^1(V)=  \RR \ell_{Q} \oplus H.\end{equation}
	The symmetric bilinear  form 
	$$q(x,y)=\sigma(x)\cdot y \cdot  \ell_{\mathbf C}, \quad x,y\in \Pi^1(V),$$
	is by Theorem~\ref{thm:AF} positive semi-definite on $H$.  Note that $x\in H$ if and only if $q(x,\ell_{Q})=0$. 
	
	Suppose that  $x\in \Pi^1(V)$ satisfies (a).
	In other words, $x\in H$ and $q(x,x)=0$. Then, the  Cauchy--Schwarz inequality implies  $q(x,y)=0$ for every $y\in H$.  Hence, for any  number $\alpha$ and $y\in H$, the equality $ q(x,\alpha \ell_{Q} + y)=0$ holds.
	In light of  the decomposition \eqref{eq:decomp}, the latter is equivalent to 
	\begin{equation}\label{eq:ker-q} 0= q(x,z) =(-1)^n \sigma(x\cdot \ell_\mathbf{C})\cdot z \quad \text{ for all } z\in \Pi^1(V).\end{equation}
	It follows from Poincar\'e duality (Theorem~\ref{thm:Poincare}) that $x\cdot \ell_\mathbf{C}=0$.
\end{proof}

We are now ready to prove the main result of this section.

\begin{proof}[Proof of Theorem~\ref{thm:equality}] We prove the theorem by induction on $n$. The base case $n=2$  was established in Lemma~\ref{lemma:dim2}. Now, assume $n>2$ and that the theorem holds in dimension $n-1$.
	Choose a euclidean inner product to identify $V$ with $\RR^n$ such that $\vol$ is the Lebesgue measure.  
	Suppose that  $x\in \A^1(E)$ satisfies \eqref{eq:equality}. Since  $\sigma(\ell_{Q}) \cdot \ell_{Q} \cdot \ell_{\mathbf C}\neq 0$, Lemma~\ref{lemma:HLHR} yields $x\cdot \ell_\mathbf{C}=0$.
	
	For  $ L\in E$, let $j_L\colon L^\perp \to \RR^n$ denote the inclusion and  $\pi_L=j^*_L\colon \RR^n\to L^\perp$ the orthogonal projection.  By Lemma~\ref{lemma:nc-res}, there exist centrally symmetric polytopes $ C_i^L \subseteq L^\perp$ such  that   $$j^*_L\ell_{C_i}= \ell_{C^L_i}\in  \Kc(j^*_LE).$$
	Moreover, Lemma~\ref{lemma:nc-res} implies $j^*_Lx\in \A^1(j^*_LE)$.

	Since \begin{equation}\label{eq:res-zero}0= j^*_L( x\cdot \ell_{\mathbf{C}})=  (j^*_L x)\cdot \ell_{C_1^L} \cdots \ell_{C_{n-2}^L},\end{equation}
	 Theorem~\ref{mthm:AF} applied inside $L^\perp$, yields the inequality
	\begin{equation}\label{eq:equality-L} \sigma (j^*_L x) \cdot (j^*_L x)  \cdot \ell_{C_1^L} \cdots \ell_{C_{n-3}^L} \geq 0.\end{equation}  
	Expanding 
	$\ell_{C_{n-2}}= \sum_{L\in E} \alpha_L x_L$,  using Proposition~\ref{prop:mult-res} and  the fact that the Euler--Verdier involution commutes with pullback, one obtains 
	$$ 0= \sigma(x)\cdot x\cdot \ell_{\mathbf C}=  \sum_{L\in E} \alpha_L \left(  \sigma(j^*_L x) \cdot (j^*_L x)  \cdot \ell_{C_1^L} \cdots \ell_{C_{n-3}^L}\right).$$
	Since $\alpha_L>0$, it follows that equality holds in \eqref{eq:equality-L} for each $L\in E$. Since
	$$\ell_{C_1^L} \cdots \ell_{C_{n-3}^L} \cdot \ell_{C_{n-2}^L}^2
	>0,$$
	  we can apply the  inductive hypothesis to obtain $j^*_L x=0$ for all $L\in E$. Using Proposition~\ref{prop:res}, we conclude that $x=0$. 
\end{proof} 

\begin{remark}The full strength of Lemma~\ref{lemma:HLHR}, which is a consequence of Poincar\'e duality, is not needed for the proof of Theorem~\ref{thm:equality}.  While Lemma~\ref{lemma:HLHR} is used to obtain \eqref{eq:res-zero},  the same conclusion can also be reached by  choosing $z=x_L$ in \eqref{eq:ker-q} and applying Proposition~\ref{prop:mult-res}.
\end{remark}

\section*{Acknowledgment} This work was supported by DFG grant WA 3510/4-1   within the priority program SPP 2458 Combinatorial Synergies. The author is grateful to Hendrik Süss for discussions on the Alesker product, which prompted the exploration of an intersection product in the polytope algebra.
The author also thanks Semyon Alesker and Ramon van Handel for comments on an earlier version of this paper.

\bibliographystyle{abbrv}
\bibliography{ref_papers,ref_books}

\end{document}